\documentclass[a4paper,11pt, reqno]{amsart}

\usepackage{a4wide}
\usepackage[english]{babel}
\usepackage{amsfonts, dsfont}
\usepackage{amsthm}
\usepackage[pdftex]{graphicx}
\usepackage{amsmath}
\usepackage{amsfonts}
\usepackage{amssymb}
\usepackage[normalem]{ulem}
\usepackage{hyperref}
\hypersetup{colorlinks,citecolor=red,filecolor=purple,linkcolor=blue,urlcolor=black}
\usepackage[T1]{fontenc}
\usepackage{enumerate}
\usepackage{relsize}
\usepackage{mathtools}
\usepackage{bbm}

\newcommand{\N}{\mathbb N}
\newcommand{\R}{\mathbb R}
\newcommand{\E}{\mathbb E}
\newcommand{\NN}{\mathcal N}

\renewcommand{\P}{\mathbb P}
\renewcommand{\l}{\lambda}
\newcommand{\gnl}{G_N^\lambda}

\newcommand{\1}[1]{\mathbf{1}\!_{\left\{#1\right\}}}
\newcommand{\ind}[1]{\mathbf{1}\!_{\,#1}}

\newcommand{\wpi}{\widehat{\pi}}

\newcommand{\e}{{\rm e}}

\newcommand{\vertiii}[1]{{\left\vert\kern-0.25ex\left\vert\kern-0.25ex\left\vert #1 
		\right\vert\kern-0.25ex\right\vert\kern-0.25ex\right\vert}}

\def\be{\begin{eqnarray}}
\def\ee{\end{eqnarray}}
\def\ben{\begin{eqnarray*}}
\def\een{\end{eqnarray*}}

\newtheorem{prop}{Proposition}[section]
\newtheorem{defi}[prop]{Definition}

\newtheorem*{thm1bis}{Theorem \ref{existence}'}
\newtheorem*{thm2bis}{Theorem \ref{maintheorem}'}
\newtheorem{lem}[prop]{Lemma}

\newtheorem{thm}[prop]{Theorem}
\newtheorem{rem}[prop]{Remark}
\newtheorem{cor}[prop]{Corollary}
\newtheorem*{assumption}{Domination $\&$ Convergence Assumption}

\newtheorem*{assumptionA2'}{Assumption A2'}

\newcounter{example}
\title{
Branching process and homogeneization for epidemics on  spatial random graphs 
}

\author[V.\ Bansaye]{Vincent Bansaye}
\address{{\'Ecole Polytechnique} CMAP, \'Ecole Polytechnique,  Route de Saclay
	91128 Palaiseau Cedex, FRANCE.}
\email{vincent.bansaye@polytechnique.edu}

\author[M.\ Salvi]{Michele Salvi}
\address{Dipartimento di Matematica, Universit\`a di Roma Tor Vergata, Via della ricerca scientifica 1, 00133, Rome, Italy}
\email{salvi@mat.uniroma2.it} 
\urladdr{\url{https://www.mat.uniroma2.it/~salvi/}}

\begin{document}

\maketitle

\begin{abstract}
Consider a graph where the sites are distributed in space according to a Poisson point process on $\mathbb R^n$. We study a population evolving on this network, with individuals jumping between sites with a rate which decreases exponentially in the distance. Individuals give also birth (infection) and die (recovery) at constant rate on each site. First,  we   construct the process, showing that it is well-posed even when starting from non-bounded initial conditions. Secondly, we prove hydrodynamic limits in a diffusive scaling. The limiting process follows a deterministic reaction diffusion equation. We use stochastic homogenization to characterize its diffusion coefficient as the solution of a variational principle.
The proof involves in particular the extension of a classic Kipnis--Varadhan estimate to cope with the non-reversiblity of the  process, due to births and deaths. 
This work is motivated by the approximation of epidemics on large networks and the results are extended  to more complex graphs including percolation of edges.
\end{abstract}

\medskip


 
 \medskip

\section{Introduction and main results}

Consider a graph $\mathcal G$ whose vertices $V$ are placed according to a Poisson point process on $\R^n$ with $n\geq 2$ and with edge set $E$ drawn from some distribution. Attach to each unoriented edge $\{x,y\}\in E$ a rate $r(x,y)=r(y,x)=\e^{-\|x-y\|}$, where $\|\cdot\|$ indicates the Euclidean distance. Consider individuals that perform independent random walks on $\mathcal G$ with jump rates $r(x,y)$. They also give birth to new individuals at rate $b\geq 0$ and die at rate $d\geq 0$. The main goal of the present work is to describe the limiting behaviour of this particle system under a diffusive rescaling. 

\smallskip 

The motivation for studying this kind of process comes from the analysis of real-world networks with agents moving on spatially inhomogeneous structures. Metapopulation models (or metacommunity for several species) 
aim at describing the habitat of a population as a collection of patches. Exchanges between 
two patches can depend on several features, in particular the distance. 
From the pioneering works of Levins \cite{Levins}, metapopulations  have a long story in biology and ecology.
Issues come  from   conservation of species (see e.g. \cite{BaSol,BanLamb}), evolution of dispersion (see e.g.~\cite{Hastings}), impact of fragmentation of habitats (see e.g.~\cite{Hiebeler}) and effect of heterogeneity of habitats (see e.g.~\cite{Pulliam}). While for the sake of simplicity one would consider a small number of patches, applications often ask for the study of large metapopulations.
As far as we know the literature, large metapopulations are considered 
either in a mean field approximation (see \cite{Levins}) or with a spatially explicit large structure using cellular automates and simulations
(\cite{BaSol}) or in a periodic environment.
Random networks provide a relevant mathematical framework to analyze models which  do not  fall in the mean field  approximation and that would  make explode the parameters' complexity if considered as large  explicit graphs. Rigorous works which combine  motion and demography (birth, death, infections...) on large random landscapes are rare for now. Our original interest is understanding how an epidemics would spread on such structures. As a driving example, one can consider the spread of an infection among cattle on the French network of farms, see \cite{Qi}. 
In this first work we identify the diffusive behaviour of an epidemics in its first stages, which corresponds to the classic branching process approximation for small ratio of infected individuals.
In this case $b$ represents the contamination rate and $d$ is the recovery rate for the infected population. This approximation is valid on a time window where the infected population remains  locally small compared to the population size, see e.g.~\cite{BD95,BR13, M19} for the classical mixed SIR model. 

\smallskip

From a mathematical point of view, the theory of hydrodynamic limits for interacting particle systems has a long story (see \cite{KL98} for a classic account). The scaling limit of population processes with births and deaths on simpler spatial structures has been analyzed for example in \cite{M96}, in \cite[Chapter 4]{DP06} and in \cite{CFV22}. In our setting, 
the first challenge is represented by the unboundedness of the jump rates: on the one hand a site in $V$ can have a huge number of close-by neighbours, so that the jump rate of a single individual can be arbitrarily large at that site; on the other hand there is no restriction on the number of individuals that can occupy a given site. Proving that such a process is well-posed is in itself not trivial: both classic and more recent techniques for proving existence fail to apply to our framework.
The second challenge is represented by the irregularity of the support $V$ combined with the lack of reversibility of the system, due to births and deaths. In order to study the limiting behaviour of the process, we need to gather  approaches coming from statistical mechanics and mathematical biology.
To be more precise, the way we can cope with the random geometry of the underlying graph is through stochastic homogenization and in particular the results of \cite{F22}. The theory of homogenization, first started in a deterministic context by analysts, describes how the microscopic irregularities of a medium affect the macroscopic behavior of the system. It is by now well understood how to use this technique to derive hydrodynamic limits for reversible particles systems, see e.g.~\cite{GJ08, F22b}. Yet, to our knowledge, one fundamental requirement for obtaining these results has  been the reversibility of the process. 
In our context  we need to adapt some tools to non--reversible population models, see e.g.~\cite{K81, BM15} for general references.  In particular,  a key ingredient  of this work is the extension of an inequality for the supremum of a particle process due to Kipnis and Varadhan. This estimate is required for the proof of tightness, for the identification of the limit and to show that this limit has a density with respect to the Lebesgue measure.

\subsection{Model and main results}
For some probability space $(\Omega, \P, \mathcal F)$ and $\omega\in\Omega$, let $V= V(\omega)$ be the points of a Poisson point process on $\R^n$ with $n\geq 2$ and intensity $\gamma>0$ under $\P$. Let $ E= E(\omega)=\{\{x,y\},\,x,y\in V\}$ be the set of unoriented edges between the points of $V$. We will consider at first the complete graph $\mathcal G=(V,E)$ as a support for our particle system, while in Section \ref{randomgraphs} we will discuss how to extend our results by generalizing $\mathcal G$ via bond percolation.
Given $\omega$ and a configuration of particles $\eta\in \N^{V}$ consider the transitions
\begin{align}\label{rates}
\eta\longrightarrow
\begin{cases}
\eta^{x,y}\quad
&\mbox{with rate }\eta (x)r(x,y)\\
\eta^{x,+}\quad
&\mbox{with rate }b \,\eta(x)\\
\eta^{x,-}\quad
&\mbox{with rate }  d\,\eta(x)
\end{cases}
\end{align}
where $\eta^{x,y}=\eta-\ind{x}+\ind{y}$ is the configuration obtained from $\eta$ by subtracting one particle in $x\in V$ and adding one in $y\in V$, $\eta^{x,+}=\eta+\ind{x}$ adds one particle in $x\in V$ and $\eta^{x,-}=\eta-\ind{x}$ has one particle less in $x\in V$. The positive numbers 
$$
r(x,y)=r(y,x)=\e^{-\|x-y\|}
$$ 
are the jump rates for each particle to go from point $x\in V$ to point $y\in V$, and vice-versa. For simplicity we set $r(x,x)=0$ for all $x\in V$. We let 
$$
r(x):=\sum_{y\in V:\,\{x,y\}\in E}r(x,y)
$$ 
be the total jump rate of a particle at site $x\in V$. It is not hard to show that, $\P$--almost surely, $r(x)$ is finite for every $x\in V$. The parameters $b,d\geq 0$ are the individual rate of birth and death of the particles, respectively. 
\smallskip

For a given realization $\mathcal G(\omega)$ of the graph, we introduce a probability space with measure $P^\omega$ under which we will construct our particle process. $E^\omega$ indicates the associated expectation. Let $\eta_0$ be the initial configuration of particles. 
 Our first result establishes that, for $\P$--almost every $\omega$, there exists a Markov process with jump rates given by \eqref{rates} as soon as $\eta_0$ has uniformly bounded expectation on each site.
\begin{thm}\label{existence}
	For $\P$-a.a.~$\omega\in\Omega$ the following holds.
	Let $\eta_0$ be a  random variable on  $\N^V$ such that, for some $M\in\N$, one has $\E^\omega[\eta_0(x)]\leq M$ for all $x\in V$. 
	 Then, 
	 for all $T>0$,  there exists a Markov process $(\eta_t)_{t\in[0,T]}$ with initial value $\eta_0$ and 
	  paths in the Skohorod space $\mathcal D([0,T],\N^{V})$ that satisfies the following:
	  	for functions $G$ compactly supported in $\R^n$, the process $(M^G_t)_{t\geq 0}$ defined by
	   $$
	  M^G_t=\sum_{x\in V} \eta_t(x)\, G(x)-\sum_{x\in V} \eta_0(x)\, G(x) - \int_0^t  \mathcal L f_G(\eta_s) \,{\rm d}s
	  $$
	  is a martingale. Here  $f_G:\N^V\to \R$ is the function $f_G(\eta)=\sum_{x\in V} G(x)\eta(x)$ and 
 \begin{align}\label{generator}
 \mathcal L f_G(\eta)
 =&\sum_{x,y\in V}\eta(x)r(x,y)\big(G(y)-G(x)\big)+\sum_{x\in V}\eta(x)\big(b- d\big) G(x)\,.
 \end{align}
	  
%
%
%
%
%
%
%
%
%
%

\end{thm}
We point out that Theorem \ref{existence} implies that $\mathcal L$ in \eqref{generator} coincides with the extended generator (see \cite[Definition 14.15]{D18}) of the process $(\eta_t)_{t\in[0,T]}$  on the set of functions of the form $f_G(\eta)=\sum_{x\in V} G(x)\eta(x)$ with $G$ compactly supported on $\R^n$.

\medskip

Our second result establishes the hydrodynamic limit of the process. Let $\mathcal M(\R^n)$ be the Polish space of non-negative Radon measures on $\R^n$ endowed with the vague topology. For $\pi\in \mathcal M(\R^n)$ and a continuous function $G\in\mathcal C(\R^n)$ we write $\langle\pi, G\rangle=\int_{\R^n} G(y)\,\pi({\rm d}y)$.  We consider a scaling parameter $N\in \N$ and associate to each element $\eta \in\N^{V}$ the empirical measure
$\pi^{N}=\pi^{N}(\eta)=N^{-n}\sum_{x\in V} \eta(x)\delta_{x/N}\in \mathcal {M}(\R^n)$, where $\delta_y$ represents a Dirac mass at  $y\in\R^n$. Conversely, we can recover $\eta$ from $\pi^N$ via $\eta(\cdot)=\eta(\pi^{N})(\cdot)=
N^n\pi^{N}(\cdot/N)$, so that for any fixed $N$ we may use $\pi^N$ and $\eta$ indifferently. In this work, we are interested in the regime where the motion is faster than births and deaths (resp.~of infection and recovery rates for epidemics). Thus, for a given $N$, we introduce now the  process  $\eta^N$ with sped--up motion. For $G$ compactly supported, its generator is given by (recall that $f_G(\eta)=\sum_{x\in V} G(x)\eta(x)$)
\begin{align}
	  \mathcal L^N f_G(\eta)
	  =&\sum_{x\in V}\eta(x)L^NG(x/N)+\sum_{x\in V}\eta(x)\big(b- d\big) G(x/N)\,.
\end{align}
Here
 \begin{align}\label{generatorrw}
L^NG(x/N)=\sum_{y\in V} N^2r(x,y) \big(G(y/N)-G(x/N)\big)
\end{align}
is the generator of the random walk on $V/N:=\{x/N:\,x\in V(\omega)\}$ with transition rates $N^2r(\cdot,\cdot)$.
 The associated measure-valued process is defined as 
\begin{align}\label{durian}
\pi_t^N:=\frac{1}{N^n}\sum_{x\in V} \eta_t^N(x)\delta_{x/N}\;.
\end{align}
Theorem \ref{existence} guarantees that, for all $T>0$ and fixed $N\in\N$, $(\pi^N_t)_{t\in[0,T]}$ is a well-defined Markov process with values in $\mathcal D([0,T],\mathcal M(\R^n))$, the space of measure-valued c\`adl\`ag processes.

\smallskip

For the scaling limit, we need to consider initial conditions such that the tails of $\eta_0$ are dominated by a product of (translated) Poisson distributions indexed by $V$, as precisely defined here below. 
This allows in particular to invoke the existence and characterization stated in Theorem \ref{existence}. For example, one can take configurations with a  number of particles that is a constant on each site or that is distributed as i.i.d.~Poisson random variables, or a sum of the two.
We also need that the initial conditions converge as $N$ goes to infinity.
For a given realization of the graph $\omega\in\Omega$, we make  thus
the following assumption.
 \begin{assumption} 
 The sequence of random configurations  $(\eta_0^N)_{N\in\N}$ satisfies the following:
 \begin{itemize}
 	\item[(i)] There exists $M\in \mathbb N_0$ and $\rho>0$ 
 	such that for any $N\in \N$, $x\in V$ and for any $A\subset V$ and $(n_x)_{x\in A}\in\N^{|A|}$,
 	\begin{align}\label{domiziano2}
 	P^{\omega}\Big(\forall x \in A, \, \eta_0^N (x) \geq M+n_x \Big)
 	\leq \prod_{x\in A}\Big(\sum_{j=n_x}^\infty \frac{\rho^j\e^{-\rho}}{j!}\Big).
 	\end{align}
 	\item[(ii)] There exists a bounded Borel function $\rho_0:\,\R^n\to [0,\infty)$
 	such that, for any $\mathcal C^\infty$ function with compact support $G\in \mathcal C^\infty_c(\R^n)$,
 	\begin{align}\label{singapore}
 	\lim_{N\to\infty} N^{-n} \sum_{x\in V} \eta_0^N(x) G(x/N)= \int_{\R^n}G(x)\rho_0(x)\,{\rm d}x\,
 	\end{align}
 	in $P^{\omega}$--probability.
 \end{itemize} 
  \end{assumption}


  Consider now $B(\Omega)$, the family of bounded Borel functions on $\Omega$, 
and let $\sigma^2\geq 0$ be characterized by the variational formula  
\begin{align}\label{variazionale}
\sigma^2:=\frac 12\inf_{\psi\in B(\Omega)}\E_{0}\Big[\sum_{y\in V}r(0,y)\big(y_1 +\psi(\theta_y\omega)-\psi(\omega)\big)^2\Big]\,.
\end{align}
Here $y_1$ denotes the first coordinate of $y\in\R^n$ and $\theta_y\omega$ is the environment translated by $y$ (see Section \ref{input} for the precise meaning of this). The expectation $\E_0$ is taken with respect to the Palm measure relative to the underlying Poisson point process, which can be obtained by just adding to the configuration a point  at the origin (see \cite{DVJ08} for an account on Palm measures).  Calling $I_n$ the $n$-dimensional identity matrix, we point out that $2 \sigma^2 I_n$ is the diffusion matrix  of the Brownian motion obtained by rescaling diffusively the random walk on the Poisson point process with transition rates $r(x,y)$, see e.g.~\cite{F22}.
\begin{thm}\label{maintheorem}
	For $\P$-a.a.~$\omega\in\Omega$ the following holds.
Let $(\eta_0^N)_{N\in\N}$ be a sequence of random variables on  $\N^V$ which satisfies the Domination $\&$ Convergence Assumption for some  bounded Borel function $\rho_0:\,\R^n\to [0,\infty)$.
Then the sequence of processes $\{(\pi^N_t)_{t\in[0,T]}\}_{N\in\N}$ with initial  value $\pi^N_0=\pi^N(\eta_0^N)$ converges in law in $\mathcal D([0,T],\mathcal M(\R^n))$ to the deterministic trajectory $(\rho(t,u)\,{\rm d}u)_{t\in[0,T]}$, where $\rho(\cdot,\cdot):[0,T]\times\R^n\to\R$ is the unique weak solution of the problem 
	\begin{align}\label{reactiondiffusion}
	\begin{cases}
	\quad\partial_t \rho\,=\,\sigma\Delta \rho +(b-d)\rho \\
	\rho(0,\cdot)\,=\,\rho_0\,
	\end{cases}.
	\end{align}
\end{thm}
Since the sequence of processes converges in distribution to a deterministic process, we obtain immediately the following convergence in  probability.
\begin{cor}
	Under the hypothesis of Theorem \ref{maintheorem} we have that, for all $t>0$, $G\in C_c(\R^n)$ and $\varepsilon >0$,
	\begin{align*}
	\lim_{N\to\infty}P^\omega\Big(\Big|N^{-n}\sum_{x\in V}G(x/N)\eta_t(x)-\int_{\R^n} G(x)\rho(x,t)\,{\rm d}x\Big| \geq \varepsilon \Big)=0\,.
	\end{align*} 
\end{cor}

\smallskip

\subsection{State of the art, techniques and structure of the paper} 

The rest of the paper is substantially divided into two parts corresponding to the proofs of the two main theorems. 

\smallskip

\begin{description}
	\item[Section \ref{sectionexistence}] 
Theorem \ref{existence} establishes the well-posedness of the process. As mentioned before, our model does not seem to be treated in the previous literature. 
The case $b=d=0$ corresponds to the motion of independent random walks, a simple instance of the so called zero-range process on $\mathcal G$. The existence of the zero-range process on an arbitrary countable state space was proved in the classical work \cite{L73} and then under weaker assumptions in \cite{A82}, but the techniques developed in those works do not apply in our setting. A first requirement for those constructions is that, in some sense, the rate of jump of each particle must be uniformly bounded from above, a condition that fails in our setting due to the irregularity of $V$. 
A second problem is that in \cite{L73} and \cite{A82} one must impose a restriction on the initial configuration of particles $\eta_0$. Namely, one accepts only $\eta_0$ satisfying $\sum_{x\in V}\eta_0(x)\alpha(x)<\infty$ for some function $\alpha$ such that $\sum_{y\in V} p(x,y)\alpha(y)\leq M\alpha (x)$, where $M>0$ is a given constant and $p(x,y)$ indicates the probability to go from $x$ to $y$ when the particle jumps. In our case, again because of the irregularity of $V$, this condition would not allow us to consider, for example, initial conditions with a constant number of particles on each site. Neither more recent approaches to prove existence for general particle systems on random graphs, like \cite{GR22}, cover our model, because of the unboundedness of the jump rates.

We adopt a different approach which borrows from \cite{A82} the idea of ghost particles. In Section \ref{measurevaluedprocess} we enlarge our space and consider a richer measure-valued process where, roughly put, particles are labelled and leave a ``ghost'' behind them every time they jump to a new site. To show existence of the original process, we pass through the well-posedness of the stochastic differential equation \eqref{colosseo} associated to this richer measure-valued process. In Section \ref{finitegraph} we prove the existence when we restrict the dynamics to a finite subgraph of $V$. We also pin down a key estimate of how many particles have visited a given compact set up to time $T$ in mean, making use of the ghosts (Lemma \ref{capogiro}). In Section \ref{finiteparticles} we extend the existence of the process when considering the whole infinite graph, but under the condition of having a finite number of particles at time $0$. This is achieved by showing that the range covered by the particles stays finite almost surely, see Proposition \ref{seggiolino}. Finally, in Section \ref{full}, we include in our construction also the case of an infinite number of initial particles. 

\smallskip

\item[Section \ref{inputfromhomo}] In this section we prepare some of the technical tools that are necessary for the proof of Theorem \ref{maintheorem}.
The operator $L^N$ can be thought of as a discretization of the operator $\sigma^2\Delta$. For a given $G\in\mathcal C_c^\infty(\R^n)$, though, some difficulties arise if one tries to prove directly the convergence of $L^NG$, due to the possible lack of regularity of this last object. To overcome the problem, one wants to substitute $G$ by a regularized version $G_N^\l$ for which  $L^NG^\l_N$ directly yields the expected limit $\sigma^2\Delta G$. 
This procedure, introduced in \cite{JL08} in the context of hydrodynamic limits and further developed in \cite{GJ08}, requires results from stochastic homogenization theory. First of all, in Section \ref{input}, we prove that indeed we are allowed to use the homogenization machinery elaborated in \cite{F22}. In Section \ref{correctedempiricalmeasure} we introduce $G_N^\lambda$ and recall some results appearing in \cite{F22}, see Lemma \ref{homogenization}. They provide bounds in norm for $G_N^\lambda$ and its convergence to $G$ in $L^1$ and in $L^2$. In Lemma \ref{tiramisu} we also control the $L^2$ norm of $L^NG$, which we need later on because of the non-conservative nature of our particle system.

\smallskip

\item[Section \ref{sectionKV}] One of the main technical ingredients for proving the hydrodynamic limit of the sequence of processes $\{(\pi^N_t)_{t\in[0,T]}\}_{N\in\N}$ is the Kipnis-Varadhan estimate to control the supremum of the particle process integrated against a test function, see \cite{KV86}. The estimate in its classic form, though, is valid only for reversible processes. In Lemma \ref{KVconservative} we adapt the Kipnis-Varadhan estimate to our model without births and deaths, which is reversible but presents some issues due to the irregularity of $V$. In Lemma \ref{stima} we extend the estimate to the non-reversible setting. The idea is to look separately at each branch of the genealogical tree of the  particles in the initial configuration. The process that looks at particles of a given branch can then be dominated by another (reversible) process, to which we can apply the original Kipnis-Varadhan type of estimate. This dominating process is obtained via a percolation procedure on the particles in the initial configuration. 

\smallskip

\item[Section \ref{proofofmaintheorem}]
The strategy to prove Theorem \ref{maintheorem} follows a classical tightness and identification procedure, which relies on  the two previous sections. 
In Section \ref{l2martingale} we consider the
martingale problem and show that the process
 $M^N$ appearing in \eqref{dynkin} is an $L^2$ martingale via a truncation argument. We also prove that $M^N$ tends to $0$ in $L^2$ as $N$ tends to infinity. This enables to easily conclude the proof of tightness by Aldous' criterion.
Finally in Lemma \ref{halls} we prove that a limiting value $(\pi_t)_{t\in[0,T]}$ of the sequence $\{(\pi^N_t)_{t\in[0,T]}\}_{N\in\N}$ must have a density with respect to the Lebesgue measure and that it has to satisfy a suitable differential equation that admits a unique weak solution, cfr.~\eqref{eqlimite}. 
For simplicity of exposition the proof until this point has been elaborated for the case without deaths, $d=0$, and we conclude the section by extending the result to the general case $d>0$, see Section \ref{verydead}. 

\smallskip

\item[Section \ref{randomgraphs}] In the very last part of the paper we show that our two main theorems continue to hold if we consider a percolation procedure on the edges of the complete graph with nodes $V$. As special cases of interest for applications, we analyze the long-range percolation and scale-free percolation random graphs. We conclude in Section \ref{generalizations} with a discussion of open problems.
\end{description}

 \subsection{Notation}
For a given realization of the graph $\mathcal G=\mathcal G(\omega)$, with  $\omega\in\Omega$, we recall that $P^\omega$ is the probability measure under which we have built the process defined in Theorem \ref{existence} and $E^\omega$ is the relative expectation. 
We will make clear each time what initial distribution of particles the process is starting from, but sometimes we will further stress the initial condition with a subscript. For example, if the initial distribution of particles on $\mathcal G(\omega)$ is $\mu$, then we can write $P^\omega_\mu$.

\begin{rem}
Through most of the proofs of the paper, we will talk directly about $P^\omega$, without specifying each time that $\omega\in\Omega$ is a realization of the underlying graph sampled according to measure $\P$. 
All the processes that appear will evolve under $P^\omega$. All the claims about these processes have to be intended to be true for $\P$--almost all $\omega$, even when we do not mention it explicitly.
\end{rem}

As mentioned before, $\mathcal M=\mathcal M(\R^n)$ stands for the Polish space of non–negative Radon measures on $\R^n$ endowed with the vague topology (namely, a sequence of measures $\nu_n$ converges to a measure $\nu$ in $\mathcal M$ if $\langle\nu_n,f\rangle\to\langle\nu,f\rangle$ for all $f\in\mathcal C_c(\R^n)$). Consequently, $\mathcal D([0,T],\mathcal M(\R^n))$ indicates the space of measure-valued c\`adl\`ag processes.

\section{Existence and characterization of the process}\label{sectionexistence}

	We will prove the existence of the process just with $d=0$.
	Indeed, a positive rate of death of the particles cannot contribute to the explosion of the process in finite time (if anything, it can help prevent it). So, if the process is well-defined for $d=0$, a completely analogous construction proves that it is also well-defined for any $d>0$.

\subsection{Measure valued process}\label{measurevaluedprocess}

\smallskip

In order to prove Theorem \ref{existence} with $d=0$ we will have to consider an auxiliary process that encodes more information than $(\eta_t)_{t\in[0,T]}$ and that lives in the space of measure-valued processes. 
Let $\mathcal I:=\N\times\bigcup_{k\geq 0}\{1,2\}^k$. 
Under measure $P^\omega$, let $(\NN_{i}^{x,y})_{i\in\mathcal I,x, y\in V}$ be a collection of  Poisson point measures on $\R_+$ with intensity 
$r(x,y) \, {\rm d} t$ and recall that $r(x,x)=0$ for each $x\in V$. These Poisson point measures are chosen independent for each ordered couple $(x,y)$ and they are also independent of the initial state $\eta_0$.  Also let $(\NN_{i}^b)_{i\in\mathcal I}$  be a collection of independent Poisson point measures  on $\R_+$  with intensity   $b\,{\rm d}t$ and independent of the $\mathcal N_i^{x,y}$'s and $\eta_0$.
The interpretation is the following:
$\mathcal I$ shall be thought of as the space of labels attached to each single particle. Particles that are present at time $0$ will be just labelled with the natural numbers. If particle $i$ is present at time $t\geq 0$, we call $X_t^i\in V$ its position. Suppose a particle with label $i=n\,i_1i_2\dots i_k$, with $n\in\N$ and $k\in\N_0$ and  $i_j\in\{1,2\}$  for $j=1,\dots,k$, is at position $X_{t-}^i$ at time $t-$ and suppose that $\mathcal N_i^b(t)-\mathcal N_i^b(t-)=1$. Then particle $i$ disappears at time $t$ and is replaced by two particles with labels $n\,i_1i_2\dots i_k1$ and $n\,i_1i_2\dots i_k2$ on the same site. If instead particle $i$ is at $X_{t-}^i=x$ at time $t-$ and $\mathcal N_i^{x,y}(t)-\mathcal N_i^{x,y}(t-)=1$, then particle $i$ disappears and generates particle $i1$ at site $y$, that is, $X_t^{i1}=y$, and it leaves behind a \textit{ghost} particle on site $x$ labelled with $i$. This way of labelling the particles is commonly known as the Ulam--Harris--Neveu notation.

 Let
$$
\wpi_t=\sum_{i\in \mathcal A_t} \delta_{(i,X_t^i,a)}+\sum_{i\in \mathcal G_t} \delta_{(i,X_t^i,g)}
$$
be the measure on $\mathcal I \times V \times \{a,g\}$ keeping track of position and state of each particle. For any $i\in \mathcal I$ and 
$u\in \{a,g\}$, one has $\wpi_t(\{i\}\times V \times \{u\})\in \{0, 1\} $.
More precisely,
\begin{align}\label{alive}
\mathcal A_t:=\{ i \in \mathcal I :  \wpi_t(\{i\}\times V \times \{a\})>0\}
\end{align}
is the set of particles that are present at time $t$ and that can jump or give birth, also called {\it alive} particles, while 
\begin{align}\label{ghost}
\mathcal G_t:=\{ i \in \mathcal I :  \wpi_t(\{i\}\times V\times \{g\})>0\}
\end{align} 
is the set of ghost particles present at time $t$.

Our aim is to construct the process $(\wpi_t)_{t\geq 0}$ which satisfies $P^\omega$--a.s., on every compact set, 
\begin{align}\label{colosseo}
\wpi_t
	=\wpi_0 &+\int_0^t \sum_{i\in \mathcal A_{s-},\,y  \in V} \left(\delta_{(i1,y,a)} +\delta_{(i,X_{s-}^i,g)} -\delta_{(i,X_{s-}^i,a)}\right) \NN_{i}^{X_{s-}^i,y}({\rm d}s) \nonumber\\
	&  +\int_0^t \sum_{i\in \mathcal A_{s-}} \left(\delta_{(i1,X_{s-}^i,a)} +\delta_{(i2,X_{s-}^i,a)} -\delta_{(i,X_{s-}^i,a)}\right) \NN_{i}^{b}({\rm d}s)
\end{align}
where, for $i\in \mathcal A_t$, $X_t^i$ is the unique element $x \in V$ such that
$\wpi_t(\{i\}\times \{x\}\times \{a\})=1$. The initial configuration $\wpi_0 $ might be random under $P^\omega$.

For a Borel set $A\subseteq\R^n$, we also let
\begin{align}\label{rabarbaro}
\pi_t(A):=\wpi_t(\mathcal I\times (A\cap V)\times \{a\})
\end{align}
be the total number of alive particles in $A$ at time $t$. We will see that $(\pi_t)_{t\in [0,T]}$ corresponds to the measure-valued process $(\pi_t^N)_{t\in [0,T]}$  introduced in \eqref{durian} with $N=1$.

\smallskip

We will show the existence of the process $(\eta_t)_{t\in [0,T]}$ by constructing the richer process $(\wpi_t)_{t\in [0,T]}$ in three steps. First, in Section \ref{finitegraph}, we will show the existence of an analogous process restricted to a finite graph. We will then build on this to extend the existence of the process on the infinite graph, but only when the initial configuration has a finite number of particles, see Section \ref{finiteparticles}. Finally in Section \ref{full} we will conclude with the existence of $(\wpi_t)_{t\in [0,T]}$ under the conditions of Theorem \ref{existence}. 

\subsection{Existence of the process on a finite graph}\label{finitegraph}
In this section we deal with a version of the process $(\wpi_t)_{t\in [0,T]}$ for which the underlying spatial point process is restricted to a finite number of points.  
Fix a bounded set $B\subset \R^n$. The process $(\wpi^B_t)_{t\in[0,T]}$
is defined as the strong solution of a stochastic differential equation whose jumps are represented by the Poisson point measures introduced in Section \ref{measurevaluedprocess}. Analogously to \eqref{alive} and \eqref{ghost}, let  $\mathcal A_t^B:=\{ i \in \mathcal I :  \wpi^B_t(\{i\}\times V\cap B \times \{a\})>0\}$ and $\mathcal G_t^B:=\{ i \in \mathcal I :  \wpi_t^B(\{i\}\times V\cap B\times \{g\})>0\}$. For $i\in \mathcal A_t^B$, we write $X_t^i$ the location of particle $i$ at time $t$, that is the unique point $x\in V$ such that $\wpi_t^B(\{i,x,a\})=1$. Then $(\wpi^B_t)_{t\in[0,T]}$ is defined via 
\begin{align}\label{skype}
\wpi_t^B
=\wpi_0^B &+\int_0^t \sum_{i\in \mathcal A_{s-}^B,\, y\in V\cap B\,} \left(\delta_{(i1,y,a)} +\delta_{(i,X_{s-}^i,g)} -\delta_{(i,X_{s-}^i,a)}\right) \NN_{i}^{X_s^i,y}({\rm d}s) \nonumber\\
&+\int_0^t \sum_{i\in \mathcal A_{s-}^B} \left(\delta_{(i1,X_{s-}^i,a)} +\delta_{(i2,X_{s-}^i,a)} -\delta_{(i,X_{s-}^i,a)}\right) \NN_{i}^{b}({\rm d}s)\,,
\end{align}
where the initial configuration $\wpi^B_0$ is a point measure on $ \mathcal I\times (V\cap B)\times \{a,g\}$, possibly random under $P^\omega$. Notice that we do not impose any restriction on the initial configuration at this stage. 

Let us justify that there exists a unique solution to this stochastic differential equation.
Since we deal with a countable discrete state space, the existence of such a solution can be shown just by constructing a stochastic process with a classical inductive scheme, where the successive jumps are given by the Poisson point measures.
This is a strong Markov process which is well defined until the potential accumulation point of the jumps (if explosion occurs) and it is the solution of \eqref{skype} until the time of explosion by construction. 
We just need to prove that explosion does not occur almost surely.
For that purpose, let us introduce the projection of the process on the last two coordinates 
$$
Z_t^B(K,u):=\wpi_t^B(\mathcal I, V\cap K,u)\qquad K\subseteq \R^n,\,u\subseteq\{a,g\}\,.
$$
$Z_t^B(K,u)$ counts the alive or ghost particles in $K$ at time $t$. 
Under $P^\omega$, $(Z_t^B)_{t\in[0,T]}$ is a multi-type branching process with a  finite number of types (that is, $(V\cap B)\times \{a,g\}$) and bounded reproduction mean:
at rate $r(x,y)$ each particle of type $(x,a)$ is replaced by two particles of types $(x,g)$ and $(y,a)$. At rate $b$ each particle of type $(x,a)$ creates a new particle of type $(x,a)$. The particles of type $(\cdot,g)$ do not evolve. 
Using classical  first moment estimates of the branching process $(Z_t^B)_{t\in[0,T]}$
we obtain non-explosivity of the process, so $(\wpi^B_t)_{t\in[0,T]}$ is well defined for any positive time $T>0$, see e.g. \cite{BM15}. Actually, we will need more quantitative estimates on the first moment of $Z_t^B$ for the limiting procedure in the next section, in particular the dependance on the transition rates. These estimates are given in the next lemma, using the harmonic function of the branching process, which is here constant in space.

\begin{lem}\label{capogiro}
	Let $b>0$.
	Consider two compact sets $K,B\subset \R^n$ with $K\subseteq B$. Take a (possibly random) initial configuration $Z_0^B$ such that $Z_0^B(B,\{g\})=0$ and, for some $M>0$, $E^\omega[Z_0^B(x,\{a\})]\leq M$ for all $x\in B\cap V$. Then it holds, for all $T>0$, 
	\begin{align}\label{ballaballaballa}
	E^\omega[Z_{T}^B(K,\{a,g\})]
		\leq C_K M\e^{bT},
	\end{align}
where $C_K= \sum_{x\in K\cap V}(b^{-1}r(x)+1)$. In particular $C_K$ does not depend on $B$.  
\end{lem}
\begin{rem}\label{octo}
In the lemma we have assumed $b>0$ to make sense of the constant $C_K$. The same proof in the case $b=0$ yields a bound in \eqref{ballaballaballa} of the form $M(\sum_{x\in K\cap V} r(x)+1)(T\vee 1)$. Notice though that the quantity $E^\omega[Z_{T}^B(K,\{a,g\})]$ can only increase as $b$ grows, so that the bound appearing in \eqref{ballaballaballa} holds for any $b>0$ when the birth rate is equal to $0$.
\end{rem}
\begin{proof} 
	We define the matrix $M_t^B$ with entries indexed by the types 
	\begin{align*}
	M^B_t\big((x,u) ,\,(y,v)\big)
		=E^\omega_{\delta_{(x,\{u\})}}\big[Z^B_t(y,\{v\})\big]\qquad x,y\in B\cap V,\,u,v\in\{a,g\}\,,
	\end{align*}
	so that $M^B_t\big((x,u) ,\,(y,v)\big)$ indicates the mean number of particles of type $(y,v)$ present at time $t$ if we started with a unique particle of type $(x,u)$ at time $0$.
	This matrix is the first moment semigroup associated to a branching process
and thus coincides with $\exp(tA)$, where  $A=A(B)$ 	is a finite matrix given by
\begin{align*}
	A=
	\begin{pmatrix}
	A_1 & A_2 \\
	0         & 0
	\end{pmatrix}
	\end{align*}
	with the following blocks: $A_1$ is the submatrix accounting for the evolution of $a$--particles, that is, for $x\neq y\in {B\cap}V$, we have $A_{1}(x,x)=b-r^B(x)$ with $r^B(x):=\sum_{z\in B\cap V}r(x,z)$ and $A_{1}(x,y)=r(x,y)$. $A_{2}$ accounts for the generation of $g$--particles from $a$--particles, that is, $A_{2}(x,y)=r(x,y)$. The two lower blocks have all the entries equal to $0$, since $g$--particles neither move nor generate other $g$--particles.
	To see that, one may use Kolmogorov forward equation or  apply the differential equation \eqref{skype}  to $(\mathcal I,y,v)$ with $\wpi_0^B=\delta_{\{1,x,\{u\}\}}$  and take expectation on both sides to get
	$M^B_t\big((x,u) ,\,(y,v)\big)=
		  \delta_{(x,u)=(y,v)}+\int_0^t M_sA\big((x,u),\,(y,v)\big)\,{\rm d}s.$

We compute now $\exp(At)$:
\begin{align*}
	A^k=
	\begin{pmatrix}
	A_1^k & A_1^{k-1}A_2 \\
	0         & 0
	\end{pmatrix}
	\quad\Longrightarrow\quad
	\e^{At}=
	\begin{pmatrix}
	\e^{A_{1}t} &  \sum_{k\geq 1}\frac{A_{1}^{k-1}A_{2}}{k!}t^k \\
	0         & Id
	\end{pmatrix}
\end{align*} 
where $Id$ is the identity matrix. Call $z_{0}=(M,M,\dots,M\,;\,0,0,\dots,0)$ the configuration having  $M$ alive particles on each site of $B\cap V$ and no ghost particles.
	Then, for any initial configuration $Z_0^B$ with less than $M$ alive particles per site in $B\cap V$ in average and no ghosts as in the hypothesis of the lemma, it holds
	\begin{align*}
	E^\omega\big[Z_{T}^B(K,\{a,g\})\big]
		=E^\omega\big[Z_{0}^B\e^{AT}(K,\{a,g\})\big]
		=E^\omega[Z_{0}^B]\e^{AT}(K,\{a,g\})
		\leq z_0\e^{AT}(K,\{a,g\})\,.
	\end{align*}
		Since $\bar 1=(1,1,\dots,1)$ is an eigenvector for the matrix $A_{1}$ with eigenvalue $b$ (that is, $\bar 1A_{1}=(b,b,\dots,b)$) and since 
	$$
	\bar 1\sum_{k\geq 1}\frac{A_{1}^{k-1}A_{2}}{k!}t^k
		=\sum_{k\geq 1}\frac{b^{k-1}(1,\dots, 1)A_{2}}{k!}t^k
		=\sum_{k\geq 1}\frac{b^{k-1}}{k!}t^k(r^B(x))_{x\in B}
		=b^{-1}(\e^{bt}-1)(r^B(x))_{x\in B}\,,
	$$
	it holds
	\begin{align*}
	z_0\e^{At}=M(\e^{bt},\dots,\e^{bt}\,;\,b^{-1}(\e^{bt}-1)r(\cdot),\dots,b^{-1}(\e^{bt}-1)r(\cdot))
	\end{align*}
	and we can conclude that
	\begin{align*}
	E^\omega[Z_{T}^B(K,\{a,g\})]
		\leq M\e^{bT}\#\{K\cap V\} + M b^{-1}\e^{bT}\sum_{x\in K\cap V}r(x)\,,
	\end{align*}
	which  implies \eqref{ballaballaballa}.
\end{proof}

\subsection{Existence of the process on the infinite graph with a finite number of initial particles}\label{finiteparticles}

In this section we want to show  that the process $(\wpi_t)_{t\geq 0}$ described in \eqref{colosseo} is well defined when we start with a configuration with a finite number of particles. We will show that  $(\wpi_t)_{t\geq 0}$ can be in fact obtained as the limit  
of the process $(\wpi_t^{B_N})_{t\geq 0}$ introduced in Section \ref{finitegraph}, where $B_N$ is the $n$-dimensional box $[-N,N]^n$.

\smallskip

Consider the process $(\wpi_t^{B_N})_{t\in[0,T]}$  introduced in Section \ref{finitegraph} up to time $T>0$. We want to show now that this process ``stabilizes'' as $N$ tends to infinity. That is, suppose to use the same source of randomness (i.e.~the same realization of the Poisson processes $\mathcal N_i^{x,y},\mathcal N_i^{b}$) to construct the process $(\wpi_t^{B_N})_{t\in[0,T]}$ for all different $N$'s. Then,   with $P^{\omega}$--probability $1$, there exists $N_0\in\N$ such that, for all $N\geq N_0$,
\begin{align}\label{siegeauto}
(\wpi_t^{B_N})_{t\in [0,T]}\equiv(\wpi_t^{B_{N_0}})_{t\in [0,T]}\,.
\end{align}
To this end, we first of all prove that the progeny of a finite number of particles remains into a finite region up to time $T>0$ with probability $1$ as $N\to\infty$.

\smallskip

For an initial configuration of alive and ghost particles $z_0$, define the maximal displacement at time $T$ as
\begin{align*}
R_N(z_0,T):=\sup_{t\in[0,T]}\sup_{i\in \mathcal A_t}\|X_t^{N,i}\|\,,
\end{align*}
where $X_t^{N,i}$ is the position of particle $i$ at time $t$ in the process $(\wpi_t^{B_N})_{t\in [0,T]}$.

	\begin{prop}\label{seggiolino}
		Consider
		an initial configuration $z_0$ with a finite number of alive particle. Then, 
		$P^\omega$--almost surely, there exists $Q>0$ such that $R_N(z_0,T)\leq Q$ for all  $N\in\N$.
\end{prop}
\begin{proof}
	We first consider $z_0$ to be constituted of a unique alive particle, labelled with $1$. Without loss of generality we can imagine particle $1$ to start at the origin.
	Abbreviate $R=R_N(z_0,T)$ and consider $(\wpi_t^{B_N})_{t\in[0,T]}$ for any $N$. For $M>0$, we can bound
	\begin{align}\label{bruco}
	P^\omega(R>M)
	=\sum_{\ell=0}^\infty P^\omega(R>M\,|\,E_{2\ell+1})P^\omega(E_{2\ell+1})
	\end{align}  
	with 
	$$
	E_\ell=E_\ell(T):=\{\mbox{particle 1 had $\ell-1$ descendants up to time $T$}\}\,.
	$$ 
	By descendant of particle $1$ we mean a particle with label starting by $1$ and that was generated via a birth event (so we do not count the particles whose label start by $1$ that were generated with a change of label due to a jump event).
	We are considering only odd integers $2\ell+1$ since each time a particle disappears it generates two new particles.
	For $\ell\in\N$, the quantity $P^\omega(E_\ell)$ is clearly dominated by $P^\omega(E_\ell^+)$, with
	$$
	E_\ell^+:=\{\mbox{particle 1 had at least $\ell-1$ descendants up to time $T$}\}\,.
	$$
	The number $Z_t^{B_N}(\R^n,\{a\})$ of alive particles at time $t$ follows a $\N$-valued Markov process starting in $1$ and that goes from $k$ to $k+1$ with rate $kb$. Let $(e_k)_{k\in\N}$ be independent exponential random variables under $P^\omega$ with  $E^\omega[e_k]=(bk)^{-1}$ and let $S_\ell:=\sum_{k=1}^\ell e_k$. We bound, for all $\theta>0$,
	\begin{align}\label{colazione}
	P^\omega\big(S_\ell\leq T\big)
	\leq \e^{\theta T}E^\omega[\e^{-\theta S_\ell}]
	= \e^{\theta T}\prod_{k=1}^\ell\Big(1-\frac \theta{kb+\theta}\Big)
	\leq  \e^{\theta T-\sum_{k=1}^\ell\frac \theta{kb+\theta}},
	\end{align} 
	where for the first inequality we have exploited the exponential Markov inequality, while for the second passage we have used the independence of the $e_k$'s and the formula for the moment generating function of the exponential distribution.

Notice that if the total number of descendants of particle $1$ at time $T$ is $2\ell$, then $Z_t^{B_N}(\R^n,\{a\})=\ell+1$. A (non-optimized) choice of $\theta=4b$ in \eqref{colazione} yields therefore
\begin{align}\label{babu}
	P^\omega(E_{2\ell+1}^+)
		=P^\omega(Z_t^{B_N}(\R^n,\{a\})\geq n+1)
		=P^\omega\Big(S_{\ell+1}\leq T\Big)
		\leq \e^{4b T-\sum_{k=1}^{\ell+1}\frac 4{k+4}}
		\leq C \ell^{-4}
\end{align}
for $\ell$ sufficiently large and some universal constant $C>0$. 

\smallskip

We move to the analysis of the term $P(R>M\,|\,E_\ell)$ in \eqref{bruco}. Abandoning for a moment the Ulam--Harris--Neveu notation, let us look at the descendants of particle $1$ and just label them $2,3,4,...$ in chronological order of birth (particles $2j$ and $2j+1$ are born in the same instant, for all $j$). 
Let $x_k\in V$ and $t_k\in[0,T]$ be the site and the time where the $k$-th particle was born and let $(X_t^k)$ be its trajectory while alive. Let $R_k$ be the maximal displacement of particle $k$, that is, $R_k:=\sup_{t\in[t_k,T]}\|X_t^k-x_k\|$.  
	We observe that
	\begin{align}\label{libellula}
	P^\omega(R>M\,|\,E_\ell)
		&=P^\omega\big(\exists k\in\{1,\dots,\ell\}:\, R_k>M/\ell \mbox{ and }R_j<M/\ell \mbox{ for all $j<k$}\,|\,E_\ell\big)\nonumber\\
		&\leq \sum_{k=1}^\ell P^\omega\big(R_k>M/\ell\,\big|\, A_{k,\ell}\big),
	\end{align}
	with $A_{k,\ell}:=\{R_j<M/\ell \mbox{ for all $j<k$}\}\cap E_\ell$.
	We have therefore to study the probability that $(X_t^k)$ left the ball $B_{M/\ell}(x_k)$ before time $T$ knowing that the first $k-1$ particles had a displacement smaller than $M/\ell$. We point out that, under $A_{k,\ell}$, we have that $B_{M/\ell}(x_k)$ is completely contained in $B_M$: it follows that all $x\in B_{M/\ell}(x_k)$ have $r(x)<C\log M$ by Lemma \ref{tyran} (see here below) and that $B_{M/\ell}(x_k)$ contains at most $C M^n\log M /\ell^n$ points of $V$ (this follows from item (ii) in the proof of Lemma \ref{tyran}).
	
Let 
	\begin{align*}
	\tau_k:=\inf \big\{t\in[t_k,T]:\, X_t^k\not\in B_{M/\ell}(x_k)\big\}\,.
	\end{align*}
	We decompose the event $\{\tau_k\leq T\}=\{\tau_k'<\tau_k\leq T\}\cup\{\tau_k\leq T,\tau_k'\}$ with 
	$$
	\tau_k':=\inf\big\{t\in[t_k,T]:\,\|X_t^k-X_{t-}^k\|\geq \sqrt{M/\ell}\big\}
	$$
	the first time that particle $k$ makes a jump longer than $\sqrt{M/\ell}$. 
	Under $A_{k,\ell}$, the event that $(X_t^k)$ makes a jump of length larger than $\sqrt{M/\ell}$ inside $B_{M/\ell}(x_k)$ has rate smaller than $\e^{-\sqrt{M/\ell}}$ times the number of the points in $B_{M/\ell}(x_k)$. It follows that
	\begin{align}\label{coccinella}
	P^\omega(\tau_k'<\tau_k\leq T\,|\,A_{k,\ell})
	\leq P^\omega(\xi\leq T)
	\leq cT\e^{-\sqrt{M/\ell}} M^n\log M /\ell^n 
	\end{align}
	for some universal constant $c>0$, where $\xi$ is an exponential random variable with parameter $C\e^{-\sqrt{M/\ell}} M^n\log M /\ell^n$.
	On the other hand, the event $\{\tau_k\leq T,\tau_k'\}$ implies that $X^k$ has performed more than $\sqrt{M/\ell}$ jumps before time $T$. Remember that under $A_{k,\ell}$ each jump has rate smaller than $C\log M$ by Lemma \ref{tyran}. Hence, if $Y$ is a Poisson random variable of parameter $CT\log M$,
	\begin{align}\label{ragno}
	P^\omega(\tau_k\leq T,\tau_k'\,|\,A_{k,\ell})
		\leq  	P^\omega(Y>\sqrt {M/\ell})
		\leq \e^{-CT\log M} \big((M/\ell)^{-1/2}\e CT \log M\big)^{\sqrt{M/\ell}}
	\end{align}
	where the last bound holds for $M/\ell$ sufficiently large, for example when $\ell\leq\sqrt M$, see e.g.~\cite[Exercise 2.3.3]{V18}.
	Continuing from \eqref{libellula}, bounds \eqref{coccinella} and \eqref{ragno} together yield, for $M$ sufficiently large and $\ell\leq\sqrt M$,
	\begin{align}\label{lombrico}
	P^\omega(R>M\,|\,E_\ell)
	&\leq \sum_{k=1}^\ell P^\omega(\tau_k\leq T\,\big|\, A_{k,\ell})
	\leq c_1 \ell \e^{-c_2\sqrt{M/\ell}}\,.
	\end{align}

	Going back to \eqref{bruco} and using \eqref{babu} and \eqref{lombrico}, we finally have
	\begin{align*}
	P^\omega(R>M)
	&\leq \sum_{\ell=0}^{\sqrt M}P^\omega(R>M\,|\,E_{2\ell+1})
	+\sum_{\ell=\sqrt M}^\infty P^\omega(E_{2\ell+1}^+)\\
	&\leq c_1 M\e^{-c_2 M^{1/4}}+ C M^{-3/2}.
	\end{align*} 
	This quantity is summable in $M$, which implies the claim by the Borel-Cantelli lemma for a single initial particle. 
	The argument can be easily generalized to any finite number of initial particles.
\end{proof}

\begin{lem}\label{tyran}
	There exists $C>0$ such that, for $\P$-a.a.~$\omega$, the following holds: there exists $\bar N=\bar N(\omega)$ such that $\forall N\geq \bar N$ one has
	$$
	\max_{x\in B_N\cap V}r(x)<C\log N\,.
	$$
\end{lem}

\begin{proof}
	We will use the two following trivial facts about Poisson point processes. Recall that $B_N=[-N,N]^n$. Let $(\mathcal B_i)_{i=1,\dots,(2N)^{n}}$ be a collection of disjoint (up to their border) volume-$1$ cubes covering $B_N$ and let $\mathcal C_j:=\{x\in B_{j+1}\setminus B_j\}$, for $j\in\N$, be the $j$-th square-crown around the origin. Then there exist constants $c_1,c_2>0$ only depending on the dimension $n$ such that, for $\P$-a.a.~$\omega$,
	\begin{itemize}
		\item[(i)] there exists $ N_1= N_1(\omega)$ such that for all $N\geq N_1$ 
		$$
		\#\{x\in \mathcal C_N\cap V\}\leq c_1 N^{n-1}\,;
		$$
		\item[(ii)] there exists $\bar N_2=\bar N_2(\omega)$ such that for all $N\geq \bar N_2$ and for all $i=1,\dots, (2N)^{n}$  
		$$
		\#\{x\in \mathcal B_i\cap V\}\leq c_2\log N\,.
		$$
	\end{itemize}
	Both facts can be checked by using classic concentration inequalities for Poisson random variables around their mean and then the Borel-Cantelli lemma.

	\smallskip
	
		Take $N\geq \max\{ N_1 , N_2\}$ and write, for $x\in B_N$,
	\begin{align}\label{brie}
	r(x)
	&=\sum_{y\in B_{2N}\cap V}\e^{-\|x-y\|}+\sum_{y\in B_{2N}^c\cap V}\e^{-\|x-y\|}\,.
	\end{align}
	For the first sum we divide $B_{2N}$ into $\mathcal B_i$'s as for item (ii) above, with $i=1,\dots, (2N)^n$. Notice that, for all $k\in\N$, there are less than $c_3 k^{n -1}$ such boxes at distance $k$ from $x$, for some $c_3>0$ that only depends on the dimension $n$. Furthermore, in each of these boxes there are at most $c_2\log(2N)$ vertices by (ii). Hence it holds
	\begin{align}\label{camembert}
	\sum_{y\in B_{2N}\cap V}\e^{-\|x-y\|}
	\leq 	\sum_{k=0}^{2N}c_3\e^{-k}k^{n-1} \cdot c_2\log (2N)
	\leq c_5\log N\,.
	\end{align}
	For the second sum in \eqref{brie} we use item (i) and bound
	\begin{align}\label{savarin}
	\sum_{y\in B_{2N}^c\cap V}\e^{-\|x-y\|}
	= \sum_{k=N}^\infty \sum_{y\in\mathcal C_k\cap V}\e^{-\|x-y\|}
	\leq \sum_{k=N+1}^\infty c_1 k^{n-1} \e^{-(k-N)}
	\leq c_6
	\end{align}
	for some $c_6>0$. Putting \eqref{camembert} and \eqref{savarin} into \eqref{brie} gives the result.
\end{proof}

\begin{cor}[Corollary of Proposition \ref{seggiolino}]\label{giardinetti}
	Consider a compact set $Q\subset \R^n$. Let $Z_0\in \N^V\times \N^V$ be an  initial configuration  such that $Z_0(x,\{a,g\})=0$ for all $x\not\in Q\cap V$ and $E^\omega[Z_0(x,a)]\leq M$ for all $x\in Q\cap V$.
	Then, $P^\omega$--a.s., for every $I\in\mathcal I$, $K\subset \R^n$ compact, $U\subseteq \{a,g\}$ and $t>0$, the following limit exists:
	$$
	\wpi_t(I,K,U):=\lim_{N\to\infty} \wpi_t^{B_N}(I,K,U)\,.
	$$	
	Furthermore, the following holds:
	\begin{itemize}
		\item[(i)]  The measure $(\wpi_t)_{t\geq 0}$  verifies equation \eqref{colosseo}, where the two sides are finite measures and coincide on $\R^n$.
		\item[(ii)] Defining for all compact sets $K\subset \R^n,$ for all $u\subseteq\{a,g\}$ and for all $t\geq0$		
		$$
		Z_t(K,u)
		:=\wpi_t(\mathcal I, K,u)\,,
		$$
		one has, for all $T>0$,
		\begin{align}\label{sangiovanni}
			E^\omega[Z_{T}(K,\{a,g\})]
		\leq C_KM\e^{bT}
		\end{align}
		where $C_K= \sum_{x\in K}(b^{-1}r(x)+1)$ (see Remark \ref{octo} for the case $b=0$). 
	\end{itemize}
\end{cor}
\begin{proof}
	The existence of $\wpi_t$ follows immediately from Proposition \ref{seggiolino}, since it implies that \eqref{siegeauto} holds $P^{\omega}$--a.s.~for all $T>0$. In particular, \eqref{siegeauto} and the fact that $\wpi_T^{B_N}(\mathcal I, \R^n,\{a,g\})<\infty$ almost surely (which follows by \eqref{ballaballaballa}) imply that  $\wpi_T(\mathcal I, \R^n,\{a,g\})<\infty$ almost surely. We let $\mathcal A_s:=\lim_{N\to\infty}\mathcal A_s^{B_N}$ for almost every realization of the process.
	
	\smallskip
	
	For item (i), we first notice that $(\wpi^{B_N}_t)$ satisfies \eqref{skype} with $B=B_N$.  $P^\omega$--a.s., for all bounded test functions $f$ with support on some set $C\subset \R^n$ and for all $t\in[0,T]$ we have
	\begin{align}\label{amatriciana}
	&\int_0^t \sum_{i\in \mathcal A_{s-},\,  y \in V} \left|f{(i1,y,a)} +f{(i,X_{s-}^i,g)} -f{(i,X_{s-}^i,a)}\right| \NN_{i}^{X_{s-}^i,y}({\rm d}s)\nonumber\\
		&\qquad\qquad\leq 3\|f\|_\infty \Big(\int_0^T \sum_{i\in \mathcal A_{s-},\,   X_{s-}^i \in C,\, y\in V}  \NN_{i}^{X_{s-}^i,y}({\rm d}s)
				+\int_0^T\sum_{i\in \mathcal A_{s-},\,   X_{s-}^i \not\in C,\, y\in C}  \NN_{i}^{X_{s-}^i,y}({\rm d}s)\Big)\nonumber\\
			&\qquad\qquad\leq 3\|f\|_\infty \wpi_T(\mathcal I, C,\{a,g\}) <\infty \end{align}
since $\wpi_T(\mathcal I, \R^n,\{a,g\})<\infty$.
		 Similarly the number of 
	births in $C$ is a.s.~controlled 
	by $\wpi_T(\mathcal I, C,\{a,g\})$:  
\begin{align}\label{carbonara}
	&\int_0^t \sum_{i\in \mathcal A_{s-}} \left|f{(i1,X_{s-}^i,a)} +f{(i2,X_{s-}^i,a)} -f{(i,X_{s-}^i,a)}\right| \NN_{i}^{b}({\rm d}s) 
		\leq 3\|f\|_\infty
		\wpi_T(\mathcal I, C,\{a,g\})\,.
	\end{align} 
	Indeed, for each newborn in $C$ there is either an active particle in $C$ or, at least, a ghost. 
To sum up, the integrals appearing on the right hand side of \eqref{colosseo} are almost surely well defined  on $\R^n$ and finite and (i) follows from $\eqref{skype}$  by letting $N$ go to infinity.
%
	
	\smallskip
	
We turn our attention to item (ii). By Fatou's lemma and Proposition \ref{capogiro},
\begin{align*}
	E^\omega[Z_{T}(K,\{a,g\})]
	\leq \liminf_{N\to\infty}E^\omega[Z_{T}^{B_N}(K,\{a,g\})]
	\leq C_KM\e^{bT}\,.
\end{align*}
This ends the proof.
\end{proof}

\subsection{Existence of the process on the infinite graph with infinitely many initial particles}\label{full}

In the previous section we have shown that the process $(\wpi_t)_{t\in [0,T]}$ is well defined as soon as the initial condition involves only a finite number of particles. 
We want to show the existence of 
$(\wpi_t)_{t\in [0,T]}$ also for initial configurations where the average number of particles on each site is bounded.

\smallskip

Consider an initial configuration of particles $Z_0\in\N^V\times \N^V$ of alive and ghost particles such that $E^\omega[Z_0(x,a)]\leq M$. For $N\in\N$ the truncated configuration $Z_{0,N}$ is obtained by considering only the particles in $Z_0$ that are inside the ball $B_N$: 
$$
Z_{0,N}(x,\cdot)=Z_0(x,\cdot)\ind{x\in B_N}\,.
$$
A central observation  is that we have monotonicity in $N$ of the process: take $N_1<N_2$ and couple the processes started in $Z_{0,N_1}$ and $Z_{0,N_2}$, call them $(\wpi_{t,N_j})_{t\geq 0}$ for $j=1,2$. Then we have $P^\omega$--a.s.
\begin{align}\label{alluvione}
\wpi_{t,N_1}(i,x,u)
	\leq \wpi_{t,N_2}(i,x,u)\qquad \forall i\in\mathcal I, \, x\in V,\, u\in\{a,g\}\,.
\end{align}
As a consequence, we have the following proposition.
\begin{prop}
	Consider a configuration of particles $Z_0\in\N^V\times \N^V$ such that $E^\omega[Z_0(x,a)]\leq M$ for all $x\in V$. Let $(\wpi_{t,N})_{t\geq 0}$ be the process on the infinite graph started in $Z_{0,N}$. 
	Then, $P^\omega$--a.s., for every $I\in\mathcal I$, $K\subset \R^n$ compact, $U\subseteq \{a,g\}$ and $t\geq 0$, the following limit exists  and is finite:
	$$
	\wpi_t(I,K,U):=\lim_{N\to\infty} \wpi_{t,N}(I,K,U)\,.
	$$
Furthermore  
		one has, for all $T>0$,
		\begin{align}\label{manzoni}
		E^\omega[
		\wpi_T(\mathcal I, K,\{a,g\})]
			\leq C_KM\e^{bT}
		\end{align}
		where $C_K= \sum_{x\in K}(b^{-1}r(x)+1)$ (for the case $b=0$ it holds the analogous of Remark \ref{octo}). 
\end{prop}
Notice that we have called the limiting process again $(\wpi_t)_{t\geq 0}$, since we have extended the definition appearing in Corollary \ref{giardinetti} to a larger set of initial conditions.

\begin{proof}
		The existence of the limit follows by the monotonicity in \eqref{alluvione}. Fix any $T>0$. We want to show now that, $P^\omega$--a.s., $\wpi_t(\mathcal I,K,\{a,g\})$ does not explode for any compact $K\subset \R^n$ and $0\leq t\leq T$. 
		Keeping in mind \eqref{alluvione}, we can use monotone convergence in $N$ to see that
	\begin{align}\label{chouchou}
	E^\omega\Big[\sup_{t\in[0,T]}\wpi_t(\mathcal I,K,\{a,g\})\Big]
		&= E^\omega[\wpi_T(\mathcal I,K,\{a,g\})]\nonumber\\
		&=\lim_{N\to\infty}E^\omega[\wpi_{T,N}(\mathcal I,K,\{a,g\})]
		\stackrel{\eqref{sangiovanni}}{\leq}C_KM\e^{bT}<\infty\,, 
	\end{align}
	where for the first equality we have used the fact that $\wpi_t(\mathcal I,K,\{a,g\})$ is also monotone in $t$, since each new event does not decrease the total number of particles in $K$.
\end{proof}
	It follows that $\sup_{t\in[0,T]}\wpi_t(\mathcal I,K,\{a,g\})$ is finite $P^\omega$--almost surely. Notice in particular that this implies that $\wpi_{t,N}(I,K,U)=\wpi_{t,M}(I,K,U)$ for all $N,M$ large enough. If this was not the case, we would have an infinite sequence of initial particles, coming from arbitrary far away, whose progeny would enter $K$ before time $T$, thus making $\wpi_T(\mathcal I,K,\{a,g\})$ explode.
\begin{cor}
	The following holds.
\begin{itemize}
	\item[(i)]  For any $t\geq 0$ and $f : \mathcal I\times \R^n \times\{a,g\}\rightarrow \R$ measurable and  compactly supported  in the second coordinate: 
\begin{align*}
			&E^\omega\Big[\int_0^t \sum_{i\in \mathcal A_{s-}, y \in V} \left\vert f(i1,y,a) +f(i,X_{s-}^i,g) -f(i,X_{s-}^i,a)\right\vert \NN_{i}^{X_{s-}^i,y}({\rm d}s) \Big]<\infty.\\
			&E^\omega\big[\int_0^t \sum_{i\in \mathcal A_{s-}} \left\vert f(i1,X_{s-}^i,a) +f(i2,X_{s-}^i,a) -f(i,X_{s-}^i,a)\right\vert \NN_{i}^{b}({\rm d}s)\Big]<\infty.
			\end{align*}
\item[(ii)] For such functions $f$ and $t\geq 0$, the following
 identity holds a.s.
		\begin{align*}
\langle \wpi_t, f\rangle
	=\langle \wpi_0, f\rangle &+\int_0^t \sum_{i\in \mathcal A_{s-},\,y  \in V} \left(f(i1,y,a) +f(i,X_{s-}^i,g) -f(i,X_{s-}^i,a)\right) \NN_{i}^{X_{s-}^i,y}({\rm d}s) \nonumber\\
	&  +\int_0^t \sum_{i\in \mathcal A_{s-}} \left(f(i1,X_{s-}^i,a) +f(i2,X_{s-}^i,a) -f(i,X_{s-}^i,a)\right) \NN_{i}^{b}({\rm d}s).
\end{align*} 
\end{itemize}
\end{cor}
\begin{proof} The first part is a consequence \eqref{amatriciana} and \eqref{carbonara}  for  bounded  functions $f$ with support on some compact set $C\subset \R^n$, together with \eqref{chouchou}  which guarantees finiteness.
We are left to show that $(\wpi_t)$ is a solution of  equation \eqref{colosseo} on any compact set, where now the initial population can be non bounded. 
	By choosing $N_0$ large, the terms involved in \eqref{colosseo} for $(\wpi_{t,N})_{t\in [0,T]}$ are all  constant for $N\geq N_0$, 
	which ends the proof.
\end{proof}

Recall that $\pi_t$ is the projection of $\wpi_t$ on alive particles, i.e.~for $B\subset \R^n$ Borel set
$\pi_t(B):=\wpi_t(\mathcal I\times (B\cap V)\times \{a\})$. 
For every $f : \R^n\rightarrow \R$ with compact support, we get
\begin{align}\label{colosseo3}
\langle \pi_t,f \rangle
	=\langle \pi_0, f\rangle +\int_0^t \sum_{i\in \mathcal A_{s-}, y \in V}  \big(f(y)  -f(X_{s-}^i)\big) \,\NN_{i}^{X_{s-}^i,y}({\rm d}s) +\int_0^t \sum_{i\in \mathcal A_{s-}} f(X_{s-}^i)  \,\NN_{i}^{b}({\rm d}s)	\,.
\end{align}
We can now justify that the generator of this process is given by \eqref{generator} and end the proof of Theorem \ref{existence}. More precisely, let us check that for all $G$ compactly supported on $\R^n$,
$$
M^G_t=\langle \pi_t,G \rangle-\langle \pi_0,G \rangle - \int_0^t  \mathcal L f_G(\eta_s) \,{\rm d}s
$$
is indeed a martingale, 
where we recall that $\eta_t(x)=\pi_t(\{x\})$ and for $f_G(\eta)=\sum_{x\in V} G(x)\eta(x)$
		\begin{align}\label{generator2}
\mathcal L f_G(\eta)
	=&\sum_{x,y\in V}\eta(x)r(x,y)\big(G(y)-G(x)\big)+\sum_{x\in V}\eta(x)\big(b- d\big) G(x).
\end{align} The fact that $\langle \pi_t,G \rangle$ is integrable is due to \eqref{manzoni}. The fact
that
$E^\omega[\int_0^t  \vert \mathcal L f_G\vert (\eta_s) {\rm d}s]$ is finite is due to 
 \eqref{amatriciana} and \eqref{carbonara}, which allows us to bound this term by $E^\omega[\wpi_t(\mathcal I, C,\{a,g\})]$. 
 Besides  
  		\begin{align*}
M^G_t
	=\langle \wpi_0, f\rangle &+\int_0^t \sum_{i\in \mathcal A_{s-},\,y  \in V} \left(f(i1,y,a) +f(i,X_{s-}^i,g) -f(i,X_{s-}^i,a)\right) \widetilde{\NN}_{i}^{X_{s-}^i,y}({\rm d}s) \nonumber\\
	&  +\int_0^t \sum_{i\in \mathcal A_{s-}} \left(f(i1,X_{s-}^i,a) +f(i2,X_{s-}^i,a) -f(i,X_{s-}^i,a)\right) \widetilde{\NN}_{i}^{b}({\rm d}s),
\end{align*} 
 where $\widetilde{\NN}_{i}$ and $\widetilde{\NN}_{i}^{b}$ are the compensated Poisson point measures. Again,  \eqref{amatriciana} and \eqref{carbonara} provide the integrability condition for stochastic $L^1$ martingale with jumps, see for example \cite{IW89}. Thus, $M^G$ inherits the  martingale property.
 This ensures that $ \mathcal L$ provides the generator for  functions of the form $f_G$.

\section{Input from homogenization} \label{inputfromhomo}

In this section we set the homogenization tools that are needed to prove the hydrodynamic limit in Theorem \ref{maintheorem}. Notice that the results we collect are mainly inherent to the environment $\omega\in\Omega$: the specific particle dynamics we are analyzing only enters in these results through the generator of the simple random walk $L^N$.

\subsection{Assumptions for homogenization on point processes}\label{input}

In \cite{F22} Faggionato proves homogenization for a wide class of random walks on purely atomic measures on $\R^n$ under some regularity assumptions for the environment, called (A1),...,(A9). Our proof of Theorem \ref{maintheorem}  relies on these homogenization results. We first state these assumptions in a simplified way, adapted to our context. 
We check then that they are indeed satisfied by our model.\\

Consider the Abelian group $\mathbb G=\R^n$. Suppose that $\mathbb G$ acts on the probability space $(\Omega,\P,\mathcal F)$ through a family of measurable maps $(\theta_g)_{g\in\mathbb G}$, with $\theta_g:\,\Omega\to\Omega$, that satisfies  the following (see $(P_1),\dots,(P4)$ in \cite{F22}):  $\theta_0$ is the identity; $\theta_g\circ \theta_{g'}=\theta_{g+g'}$ for all $g,g'\in\mathbb G$; the map $(g,\omega)\to\theta_g\omega$ is measurable; $\P\circ\theta_g^{-1}=\P$ for all $g\in \mathbb G$. The group $\mathbb G$ acts also on the space $\R^n$ as space-translations $(\tau_g)_{g\in\mathbb G}$ so that $\tau_gx=x+g$ for all $g\in\mathbb G$ and $x\in \R^n$.  Suppose to have a random purely atomic locally finite non-negative measure $\mu_\omega\in \mathcal M(\R^n)$
\begin{align*}
\mu_\omega=\sum_{x\in \hat\omega}n_x(\omega)\delta_x,\qquad n_x(\omega):=\mu_\omega(\{x\}),\qquad \hat\omega:=\{x\in\R^n:\,n_x(\omega)>0\}\,.
\end{align*}
Let $\mathbb P_0$ be the Palm measure associated to $\P$ and $\E_0$ the related expectation (see for example \cite[equation (9)]{F22} for the precise definition and \cite{DVJ08} for further properties of the Palm distribution). Finally let $r:\,(\omega,x,y)\to r(\omega,x,y)\in[0,\infty)$ be the jump rates with $r(\omega,x,x)=0$ for all $x\in\R^n$ and $\omega\in\Omega$, and $r(\omega,x,y)=0$ when $x$ or $y$ is not in $\hat\omega$. Then the nine assumptions are the following, with $\Omega_*$ some measurable, translation invariant subset of $\Omega$  with $\P(\Omega_*)=1$:

\begin{itemize}
	\item[(A1)] $\P$ is stationary and ergodic w.r.t.~$(\theta_g)_{g\in\mathbb G}$. That is, $\P\circ\theta_g^{-1}=\P$ for all $g\in \mathbb G$ and, 
	for each $A\subseteq \Omega$ such that $A=\theta_{g}A$ for all $g\in\mathbb G$, one has $\P(A)\in\{0,1\}$; 
	\item[(A2)] $0<\E[\mu_\omega([0,1)^n)]<\infty$;
	\item[(A3)] for all $\omega\in\Omega_*$ and all $g\neq g'$ it holds $\theta_g\omega\neq\theta_{g'}\omega$;
	\item[(A4)] for all $\omega\in\Omega_*$, $\mu_\omega $ is $\mathbb G$--stationary: for all $x,y\in\R^n$ and for all $g\in\mathbb G$ it holds $\mu_{\theta_g\omega}=\tau_g\mu_{\omega}$ and $r(\theta_g\omega,x,y)=r(\omega,\tau_gx,\tau_gy)$;
	\item[(A5)] for all $\omega\in\Omega_*$ and for all $x,y\in\hat\omega$ it holds $n_x(\omega)r(\omega,x,y)=n_y(\omega)r(\omega,y,x)$;
	\item[(A6)] for all $\omega\in\Omega_*$ and for all $x,y\in\hat\omega$ there exists a path $x=x_0,\,x_1,\dots,x_{n-1},x_n=y$ such that $r(\omega,x_i,x_{i+1})>0$ for all $i=0,\dots,n-1$;
	\item[(A7)] $\E_0\big[\sum_{x\in\hat\omega}r(\omega,0,x)|x|^k\big]<\infty$  for $k=0,2$;
	\item[(A8)]  $L^2(\P_0)$ is separable;
	\item[(A9)] setting $N_z(\omega):=\mu_\omega(z+[0,1)^n)$ for $z\in\mathbb Z^n$, it holds $\E[N_0^2]<\infty$ and, for some $C\geq 0$, $|Cov(N_z,N_{z'})|\leq C|z-z'|^{-1}$.
\end{itemize}

\smallskip

We prove now that these assumptions are satisfied for our model. 
\begin{lem}\label{corredo}
The complete graph $\mathcal G=(V,E)$ on a Poisson point process of parameter $\gamma>0$ in $\R^n$ with transition rates $r(x,y)=r(y,x)=\e^{-\|x-y\|}$ (with the convention $r(x,x)=0$) satisfies assumptions $(A1),\dots,(A9)$ of \cite{F22} by taking: 
	$\Omega$ the locally finite subsets of $\mathbb R^n$; $\P$  the law of a Poisson point process of parameter $\gamma>0$ in $\R^n$; $(\theta _g)_{g\in\R^n}$ the standard translations and $\mu_\omega$ the counting measure: if $\omega=\{x_i\}_{i\in\N}$ then $\theta_g\omega=\{x_i-g\}_{i\in\N}$ and $\mu_\omega(A)=\sum_{x\in\omega}\delta_x(A)$ for all measurable sets $A\subseteq\R^n$.
\end{lem}


\begin{proof}[Proof of Lemma \ref{corredo}]
	The proof is a special case of the discussion contained in \cite[Section 5.4]{F22}, we report few details here for completeness.
	Clearly in this case $\omega=\hat\omega=V(\omega)$ is the set of vertices of our graph and almost surely $n_x(\omega)=1$ for each point $x\in V$. (A1), (A2) and (A3)  clearly hold. (A4) and (A5) also come from the stationarity of the Poisson point process and from our choice of the rates. (A6) is trivial since we are considering the complete graph. For (A7) and (A8), we mention that  the Palm measure associated to the underlying Poisson point process can be obtained by just adding an additional point to the configuration at the origin. (A7) is easy to verify, while for (A8) see the comment at the end of Section 2.4 in \cite{F22}. Finally, $N_z$ is just the number of points in the box $z+[0,1)^n$, so that $\E[N_0^2]=\gamma^2+\gamma<\infty$ and the covariance appearing in (A9) is just equal to 0. 	
\end{proof}

\subsection{The Poisson equation}\label{correctedempiricalmeasure}

Fix $\omega\in\Omega$.
Recall from \eqref{generatorrw} that $L^N$ is the generator of the diffusively rescaled random walk on $V/N:=\{x/N:\,x\in V(\omega)\}$ with transition rates $N^2r(\cdot,\cdot)$. 
We think of $L^N$ as acting on functions in $L^2(\mu_N)$, where $\mu_N=\mu_N(\omega)$  is the uniform measure on $V/N$, that is 
$$
\mu_N:=N^{-n}\sum_{x\in V}\delta_{ x/ N}\,.
$$ 
We write $(\cdot,\cdot)_{\mu_N}$  and  $\|\cdot\|_{L^2(\mu_N)}$ for, respectively, the scalar  product and the norm in $L^2(\mu_N)$. Note that $L^N$ is a negative-definite symmetric operator:
for any $f,g\in L^2(\mu_N)$
$$
(f,L^Ng)_{\mu_N}=(L^Nf,g)_{\mu_N}\qquad\mbox{and}\qquad (f,-L^Nf)_{\mu_N}\geq 0\,.
$$

The following definition is justified by the fact that $L^N$ should approach in some sense the continuous operator $\sigma^2\Delta$.
\begin{defi}\label{definitiongnl}
	Given $\l>0$, $G\in C_c^\infty(\R^n)$ and $N\in\N$, we define $G^\l_N$ to be the unique element in  $L^2(\mu_N)$ such that
	\begin{align}\label{glambda}
	\l G_N^\l-L^NG^\l_N=H_N\,
	\end{align}
	where $H_N$ is the restriction to $V/N$ of the function $H=H(\l)=\l G-\sigma^2\Delta G\in C_c^\infty(\R^n)$.
\end{defi}
Notice that the introduction of $\l>0$ is just an artifice to make  $\l Id - L^N$ invertible, where $Id$ is the identity operator, and that $\l$ will be fixed and play basically no role in what follows. The idea for introducing $\gnl$ is that $L^N\gnl$ is more regular than $L^N G$ (for example inequality \eqref{regulargnl} here below might fail for a general $G$).
This regularizing procedure is associated to the so-called corrected empirical measure in the literature, see \cite{GJ08} for more comments on this.

The next lemma is where homogenization theory enters the game. We report some of the results appearing in \cite{F22}, and in particular Theorem 4.4 therein.
\begin{lem}\label{homogenization}
	Fix $\l>0$. Then for $\mathbb P$-a.a.~$\omega$ and for each $G\in C_c^\infty(\R^n)$ it holds
	\begin{align}
	&(\gnl,-L^N\gnl)_{\mu_N}\leq c(\l,G)\label{gnlinner}\\
	&\|\gnl\|_{L^1(\mu_N)}\,,\;\|\gnl\|_{L^2(\mu_N)}\leq c(\l,G)\label{gnll1}\\
	&\|L^N \gnl\|_{L^1(\mu_N)}\,,\;\|L^N\gnl\|_{L^2(\mu_N)}\leq c(\l,G)\,,\label{regulargnl}
	\end{align}
	where $c(\l,G)>0$ is a constant not depending on $N$. Furthermore
	\begin{align}
	\lim_{N\to\infty}\|\gnl-G\|_{L^1(\mu_N)}&=0 \label{eqhomogenization1}\\
	\lim_{N\to\infty}\|\gnl-G\|_{L^2(\mu_N)}&=0 \label{eqhomogenization2}\,.
	\end{align}
\end{lem}
\begin{proof}
	Thanks to Lemma \ref{corredo} we can use the results of \cite{F22}.
	Formulas \eqref{eqhomogenization1} and \eqref{eqhomogenization2} appear as formulas (54) and (55) in Theorem 4.4 of \cite{F22}. The bounds \eqref{gnll1} and \eqref{regulargnl} follow immediately, see also the proof of Theorem 4.4 of \cite{F22}. 
	The bound \eqref{gnlinner} is standard (see e.g.~the beginning of the proof of \cite[Lemma 3.1]{F10}).
\end{proof}

As a result of having to deal with a non-conservative system, in order to study the hydrodynamic limits we will also have to control  the $L^2(\mu_N)$ norm of $L^NG$.

\begin{lem} \label{tiramisu}
	Let $G\in \mathcal C_c^\infty(\R^n)$ and $n\geq 2$. Then, $\P$--a.s.,
	$$
	\lim_{N\to\infty}N^{-n}\|L^N G\|_{L^2(\mu_N)} =0\,.
	$$
\end{lem}
\begin{proof}
	First of all we bound the second moment of $\|L^N G\|_{L^2(\mu_N)} $. Call $S_G$ the support of $G$ and indicate with $NS_G$ the support blown by a factor $N$.
	We have
	\begin{align}
	\E\big[\|L^N G\|_{L^2(\mu_N)} ^2\big]
		&=\E\Big[ N^{-n}\sum_{x\in V} \Big(\sum_{y\in V}N^2r(x,y)\big(G(y/N)-G(x/N)\big)\Big)^2 \Big]\\
		&\leq 2N^{4-n}\big( (A)+(B)+(C)\big),\label{jeremy}
	\end{align}
	where
	\begin{align*}
	(A)
		&=\E\Big[ \sum_{x\in NS_G\cap V} \Big(\sum_{y\in B_{R}(x)\cap V}r(x,y)\big(G(y/N)-G(x/N)\big)\Big)^2 \Big]
		\\
	(B)
		&=\E\Big[ \sum_{x\in NS_G\cap V} \Big(\sum_{y\in B_{R}^c(x)\cap V}r(x,y)\big(G(y/N)-G(x/N)\big)\Big)^2\Big]
		\\
	(C)
		&=\E\Big[ \sum_{x\in (NS_G)^c\cap V} \Big(\sum_{y\in NS_G\cap V}r(x,y)G(y/N)\Big)^2\Big],
	\end{align*}
	where $B_R(x)$ is a ball around $x$ of radius $R=\log N^n$. We proceed by estimating separately the three parts. 
	We can easily deal with part $(B)$ thanks to  Slivnyak-Mecke theorem (see \cite[Theorem 13.3]{MW03} or \cite[Chapter 13]{DVJ08} for more general versions of the theorem), which yields
	\begin{align}
	(B)
		&\leq \|G\|_{\infty}^2\int_{x\in NS_G}\int_{y\notin B_R(x)}\Big(r(x,y)^2+r(x,y)\int_{z\notin B_R(x)}r(x,z)\,{\rm d}z\Big)\,{\rm d}y \,{\rm d}x 
		\leq c N^n \e^{-R}\label{penny}\,,
	\end{align}
	where the factor $N^n$ comes from the size of $NS_G$ and the factor $\e^{-R}$ comes from the internal integrals.
	
	 Developing the square and using again  Slivnyak-Mecke theorem, term $(C)$ becomes
	\begin{align}
	(C)
	=&\int_{x\notin NS_G}\int_{y\in NS_G} r(x,y)^2G(y/N)^2\,{\rm d}y\,{\rm d}x\nonumber\\
	&+\int_{x\notin NS_G}\int_{y\in NS_G} \int_{z\in NS_G}r(x,y)r(x,z)G(y/N)G(z/N)\,{\rm d}y\,{\rm d}z\,{\rm d}x.\label{zodiac}
	\end{align}
	Since $G\in\mathcal C^\infty_c(\R^n)$, $G$ must be Lipshitz with  Lipshitz constant, say, $K>0$. Call $d(x,A)$  the distance between $x\in\R^n$ and the border of the set $A\subset\R^n$. Noticing that $\|x-y\|\geq d(x,NS_G)+d(y,NS_G)$ if $x\not\in NS_G$ and $y\in S_G$, we see that the first double integral on the r.h.s.~of \eqref{zodiac} is smaller than
	\begin{align*}
	\int_{x\notin NS_G}\int_{y\in NS_G} \e^{-2(d(x,NS_G)+d(y,NS_G))}&\Big(K\frac{d(y,NS_G)}{N}\Big)^2\,{\rm d}y\,{\rm d}x\\
		&\leq c_1N^{-2}\int_{x\notin NS_G}\e^{-2d(x,NS_G)}N^{n-1}\,{\rm d}x
		\leq c_2 N^{d-3}
	\end{align*}
	with $c_1,c_2>0$ constants that depend on $G$.	Regarding the triple integral on the r.h.s.~of \eqref{zodiac} we can do something similar and bound it by
	\begin{align*}
	\int_{x\notin NS_G}\e^{-2d(x,NS_G)}\Big(\int_{y\in NS_G} \e^{-2d(y,NS_G)}   \Big(K\frac{d(y,NS_G)}{N}\Big)  \,{\rm d}y\Big)^2\,{\rm d}x
	\leq cN^{2n-4},
	\end{align*}
	with $c>0$ a constant depending on $G$. Plugging these two last bounds back into \eqref{zodiac} we got
	\begin{align}\label{lane}
	(C)\leq cN^{2n-4}.
	\end{align}
	
	Finally we turn our attention to 
	$(A)$. We use once more Slivnyak-Mecke theorem and a first order Taylor approximation and obtain 
	\begin{align}
	(A)
		&=\int_{x\in NS_G} \E\Big[\Big(\sum_{y\in B_R(x)\cap V} r(x,y)\Big(\sum_{i=1,\dots,n}\frac{y_i-x_i}{N} \tfrac{\rm d}{{\rm d}x_j}G(x/N)+O\big(\|x-y\|^2/N^2\big)\Big)\Big)^2\Big]\nonumber\\
		&\leq c_1N^{-2}\|\nabla G\|_\infty ^2\int_{x\in NS_G} \E\Big[\Big(\sum_{y\in B_R(x)\cap V} r(x,y)\big(\|x-y\|+\|x-y\|^2/N\big)\Big)^2\Big]\,{\rm d}x\nonumber\\
		&\leq c_2N^{-2}\|\nabla G\|_\infty ^2\int_{x\in NS_G}U(x,R) \,{\rm d}x\nonumber\\
		&\leq c_3 N^{n-2}\label{zubriscola}
	\end{align}
	where we have used the fact that 
	\begin{align*}
	U(x,R)
		:=\int_{y\in B_R(x)}\Big( r(x,y)^2\|x-y\|^2+  \int_{z\in B_R(x)} r(x,y)r(x,z)\|x-y\|\,\|x-z\|\,{\rm d}z\Big){\rm d}y
		\leq c
	\end{align*}
	for some $c>0$.

	We finally put \eqref{penny}, \eqref{lane} and \eqref{zubriscola}  back into \eqref{jeremy} to obtain that
	\begin{align*}
		\E[\|L^N G\|_{L^2(\mu_N)} ^2]\leq  cN^n\,.
	\end{align*}
	By Markov inequality we obtain now that, for all $\varepsilon>0$,
	\begin{align*}
	\P(N^{-n}\|L^N G\|_{L^2(\mu_N)}>\varepsilon)\leq c\,\varepsilon^{-2}N^{-n}\,,
	\end{align*}
	which tells us that the sequence $N^{-n}\|L^N G\|_{L^2(\mu_N)}$ converges almost completely to $0$ for $n\geq 2$ and hence almost surely. 
\end{proof}

\section{A non-conservative Kipnis--Varadhan estimate}\label{sectionKV}


Recall the Domination $\&$ Convergence Assumption and in particular \eqref{domiziano2}. For constants $\rho>0$ and $M\in\N_0$,
call $\nu_{M,\rho}(\cdot)=\nu_{M,\rho}(\omega,\cdot)$ the measure that dominates all initial conditions. That is, $\nu_{M,\rho}$ is the product measure on $\N^V$ such that its restriction on each site $x\in V$ is a Poisson random variable of parameter $\rho$ plus the constant $M\in\N$:
 	\begin{align}\label{domiziano}
\nu_{M,\rho}\Big(\prod_{x\in A}[M+n_x,\infty)\Big)
= \prod_{x\in A}\Big(\sum_{j=n_x}^\infty \frac{\rho^j\e^{-\rho}}{j!}\Big)\qquad \forall A\subset V,\, (n_x)_{x\in A}\in\N^{|A|} \,.
\end{align}

\begin{lem}\label{KVconservative}
	Consider an initial condition given by $\nu_{0,\rho}$, the product of Poisson random variables of parameter $\rho>0$. 
	Under $P_{\nu_{0,\rho}}^\omega$, let each particle perform an independent random walk on $V/N$ with generator $L^N$ (without births nor deaths) and call $(Y_t)_{t\geq 0}$ the evolution of their configuration, so that $Y_t(x) $ is the number of particles in $x\in V$ at time $t$. 
	Let $H$ be a nonnegative function on $V/N$ belonging to $L^1(\mu_N)\cap L^2(\mu_N)$ and such that $L^NH$ belongs to $L^2(\mu_N)$. 
	Then for any $T,A>0$ it holds
	\begin{align}
	P_{\nu_{0,\rho}}^\omega\Big(\sup_{0\leq t\leq T} \frac{1}{N^n}\sum_{x\in V}Y_t(\tfrac xN) H(\tfrac xN)>A\Big)
		&\leq  c(\rho,T) A^{-1} \vertiii{H}_N  \label{zapotec}
\\
	P_{\nu_{0,\rho}}^\omega\Big(\sup_{0\leq t\leq T} \frac{1}{N^n}\sum_{x\in V}Y_t(\tfrac xN)^2 H(\tfrac xN)>A\Big)&
		\leq  \tilde c(\rho,T) A^{-1}\sqrt{\vertiii{H}_N^2 +N^{2-2n}\sum_{x\in V}r(x)H(\tfrac xN)}\label{gambadilegno}
	\end{align} 
	with $c(\rho,T)=(\rho^2+\rho+T\rho)^{ 1/2}$, $\tilde c(\rho,T)^2$ a polynomial in $\rho$ and $T$ and 
\begin{align}\label{babu}
	\vertiii{H}_N^2
		:=\|H\|_{L^1(\mu_N)}^2+ N^{-n} \| H \|_{L^2(\mu_N)} \, \| L^NH  \|_{L^2(\mu_N)}\,.
\end{align}

\end{lem}

\begin{rem}
This sort of inequalities are typically carried out for all powers of the number of particles $Y_t^k$ at once, at the only cost of a constant on the r.h.s.~varying with $k$, see for example \cite[Lemma 3.2]{F10}. In our setting, though, we cannot hope for such a ``clean'' result for all values of $k$, due to the irregularity of the support $V=V(\omega)$. In the rest of the paper we  only need  $k=1$, but we bound here also the case $k=2$ for future interest.
\end{rem}

\begin{proof}[Proof of Lemma \ref{KVconservative}]
	The particle dynamics without births or deaths is reversible with respect to $\nu_{0,\rho}$. Hence, by Kipnis-Varadhan  inequality (\cite{KV86}, see also \cite[Theorem 11.1 in Appendix 1]{KL98}) we know that, for $k\geq 1$,
	\begin{align}\label{donvar}
	P_{\nu_{0,\rho}}^\omega\Big(\sup_{0\leq t\leq T} \frac{1}{N^n}\sum_{x\in V}Y_t(\tfrac xN)^k  H(\tfrac xN)>A\Big)
		\leq \frac{\rm e}{A} \sqrt{\langle g,g\rangle_{\nu_{0,\rho}}+ T\,\langle g,-{N^2\mathcal L_*}g\rangle_{\nu_{0,\rho}}}
	\end{align}
	where ${ N^2\mathcal L_*}$  is the generator of $(Y_t)_{t\geq 0}$ and $g:\N^{V}\to\R$ is given by 
	$$
	g(\eta):= \frac{1}{N^n}\sum_{x\in V} g_x(\eta) H(x/N) ,\qquad g_x(\eta):=\eta(x)^k.
	$$
	Notice that $\mathcal L_*$ corresponds to $\mathcal L$ appearing in \eqref{generator} with $b=d=0$.
	Now we calculate
	\begin{align}\label{inner}
	\langle g,g\rangle_{\nu_{0,\rho}}
		&=\frac{1}{N^{2n}}\sum_{x,y\in V} H(x/N)H(y/N)\nu_{0,\rho}[g_xg_y] \leq c_0(\rho,k)\|H\|_{L^1(\mu_N)}^2
	\end{align}
	where $c_0(\rho,k)=\E[\xi_\rho^{2k}]$ indicates the $2k$-th moment of $\xi_\rho\sim$Poisson($\rho$). Moving to the second summand under the root in \eqref{donvar}, we write
	\begin{align}\label{dirichlet}
	\langle g,-N^2\mathcal L_*g\rangle_{\nu_{0,\rho}}
		&=-N^{2-2n}\sum_{x,y\in V} H(x/N)H(y/N)\nu_{0,\rho}[g_x \, \mathcal L_*g_y]\,.
	\end{align}
Besides, we have
	\begin{align*}
	\mathcal L_*g_y (\eta)
		=\eta(y)r(y)\big((\eta(y)-1)^k-\eta(y)^k\big)
		+\sum_{z\in V}\eta(z)r(z,y)\big((\eta(y)+1)^k-\eta(y)^k\big).
	\end{align*}
For $x=y$, we get
	\begin{align*}
	\nu_{0,\rho}\big[g_x\, \mathcal L_*g_x\big]
		&=c_1(\rho,k)r(x)
	\end{align*}
	with $$c_1(\rho,k)=\E\big[\xi_\rho^{k+1}((\xi_\rho-1)^k-\xi_\rho^k)\big]+\E[\xi_\rho]\E\big[\xi_\rho^{k}((\xi_\rho+1)^k-\xi_\rho^k)\big]\leq 0\,.$$ 
	For $x\neq y$, using that $\mathcal L_*g_y(\eta)-\eta(x)r(x,y)\big((\eta(y)+1)^k-\eta(y)^k)$ is independent of $\eta(x)$ under $\nu_{0,\rho}$
	and that $\nu_{0,\rho}[\mathcal L_*f]=0$ for all $f$,
	\begin{align*}
	\nu_{0,\rho}\big[g_x \, \mathcal L_* g_y\big]
		&=\nu_{0,\rho}\Big[\eta^k(x)\cdot\eta(x)r(x,y)\big((\eta(y)+1)^k-\eta(y)^k\big)\Big]\\
		&\quad+ \nu_{0,\rho}[\eta^k]\nu_{0,\rho}\Big[\mathcal L_*(g_y)-\eta(x)r(x,y)\big((\eta(y)+1)^k-\eta(y)^k\big)\Big]\\
		&=c_2(\rho,k) r(x,y)\,.
	\end{align*}
	with $$c_2(\rho,k)=\big(\E[\xi_\rho^{k+1}]-\E[\xi_\rho^k]\E[\xi_\rho]\big) \E\big[(\xi_\rho+1)^k-\xi_\rho^k\big]\,.$$

	When $k=1$ we magically have $c_2(\rho,1)=-c_1(\rho,1)=\rho$, so that
	\begin{align*}
	\langle g,-\mathcal L_*g\rangle_{\nu_{0,\rho}}
		&=\rho N^{2-2n}\Big(\sum_{x\in V} r(x)H(x/N)^2
		-\sum_{x\ne y\in V} r(x,y) H(x/N)H(y/N)\Big)\\
	&=\rho N^{-2n}\sum_{x \in V} H(x/N) 			\Big(-N^2\sum_{y\ne x} r(x,y)\big(H(y/N)-H(x/N)\big)\Big)\\
	&=\rho N^{-n} \langle H, L^NH\rangle_{\mu_N}\\
	&\leq \rho N^{-n} \| H \|_{L^2(\mu_N)} \, \| L^NH  \|_{L^2(\mu_N)}\,.
	\end{align*}
	Putting this and \eqref{inner}  back into \eqref{donvar} together with the fact that $c_0(\rho,1)=\E[\xi_\rho^2]=\rho^2+\rho$ gives \eqref{zapotec}.
	
	When $k=2$,  explicit calculation yield $c_1(\rho,2)=-\rho(4\rho^2+8\rho+1)$ and  $c_2(\rho,2)=4\rho^3+4\rho^2+\rho$. They do not cancel out as in the case $k=1$ and as a consequence we have another term appearing from the term $\langle g,-\mathcal L_*g\rangle_{\nu_{0,\rho}}$, that is
	\begin{align*}
	\langle g,-\mathcal L_*g\rangle_{\nu_{0,\rho}}
		&\leq c_2(\rho,k) N^{-n} \| H \|_{L^2(\mu_N)} \, \| L^NH  \|_{L^2(\mu_N)} + \big|c_1(\rho,k)+c_2(\rho,k)\big|\,R
	\end{align*}
	with
	$$
	R=N^{2-2n}\sum_{x\in V} r(x)H(x/N)^2.
	$$
	Putting the pieces together as before we obtain \eqref{gambadilegno}. 
\end{proof}

Let us turn to the non-conservative case.
\begin{lem}\label{stima}
	Let $\eta_0^N$ be an initial distribution of particles whose law is dominated by $\nu_{M,\rho}$ for some $M\in\N_0$ and $\rho\geq 0$ in the sense of \eqref{domiziano2}.
	Let $H$ be a nonnegative function on $V/N$ belonging to $L^1(\mu_N)$ and $L^2(\mu_N)$.  Then there exist a constant $c_1=c(M,\rho,T)>0$ and an absolute constant $c_2>0$  such that
	\begin{align}\label{allesliebe}
	P^\omega\Big(\sup_{0\leq t\leq T} \frac{1}{N^n}\sum_{x\in V}\eta^{N}_t(x)  H(x/N)>A\Big)
		\leq A^{-1} c_1{\rm e}^{ c_2bT} \vertiii{H}_N
	\end{align}
	for all $A>0$, where $\vertiii{H}_N$ is defined in \eqref{babu}.
\end{lem}
\begin{proof}
The probability appearing in \eqref{allesliebe} can be clearly upper bounded by the probability of the same event starting with a configuration sampled with $\nu_{M,\rho}$.
As a first step, we would like to further bound the initial condition in order to have a pure product of a Poisson number of particles per site, which will allow us to use the result of Lemma \ref{KVconservative} in the following. To this end we first focus on the case $M=1$, $\rho=0$. In this case we have that the initial condition $\nu_{1,0}$ is given by a single particle on each site of $V $. Take two random variables $X,Y\sim \text{Poisson}(\log 2)$ such that
\begin{align*}
P(X+Y\geq 1)=1\,.
\end{align*}
We can dominate $\nu_{1,0}$ by the random initial condition $\bar\nu$ given by the following: the number of particles on site $x\in V$ is given by $X(x)+Y(x)$, with $X(x)\sim X$ and $Y(x)\sim Y$ and $(X(x),Y(x))_{x\in V}$ independent for different $x\in V$. Now we notice that if we want $N^{-n}\sum_{x\in V}\eta^N_t(x) H(\tfrac xN)$ to be greater than $A$, it must be that $N^{-n}\sum_{x\in V}\eta^{N,X}_t(x) H(\tfrac xN)$ is larger than $A/2$, with $\eta_t^{N,X}$ are the particles descending from initial particles ``of type X'', or $N^{-n}\sum_{x\in V}\eta^{N,Y}_t(x) H(\tfrac xN)$ has to be greater than $A/2$. So with a union bound we get
\begin{align*}
P_{\nu_{1,0}}^\omega\Big(\sup_{0\leq t\leq T} \frac{1}{N^n}\sum_{x\in V}\eta^N_t(x) H(\tfrac xN)>A\Big)
	\leq 2 	P^\omega_{\nu_{0,\log 2}}\Big(\sup_{0\leq t\leq T} \frac{1}{N^n}\sum_{x\in V}\eta^N_t(x) H(\tfrac xN)>A/2\Big)\,.
\end{align*}
It is straightforward to generalize the previous argument to the case $M\geq 1$ and $\rho\geq 0$ which yields
\begin{align*}
P_{\nu_{M,\rho}}^\omega\Big(\sup_{0\leq t\leq T} \frac{1}{N^n}&\sum_{x\in V}\eta^N_t(x) H(\tfrac xN)>A\Big)\\
	&\leq (M+1)	P^{\omega}_{\nu_{0,\rho\vee \log 2}}\Big(\sup_{0\leq t\leq T} \frac{1}{N^n}\sum_{x\in V}\eta^N_t(x) H(\tfrac xN)>\tfrac{A}{M+1}\Big)\,.
\end{align*}
From this we see that, at the cost of a constant factor depending on $M$, we can prove \eqref{allesliebe} with initial particle configuration $\nu_{0,\rho}$, where we have replaced the original $\rho$ with $\rho\vee \log 2$.

\smallskip

We use a new labelling notation for the particles, not to be confused with the one appearing in Section \ref{measurevaluedprocess}. The  individuals at time $0$ are labelled by $\N$. 
To label their descendants, we introduce the binary tree
$$
\mathcal J=\cup_{k\in\N_0} \{1,2\}^k.
$$
For $j=(j_1,\dots,j_k)\in\mathcal J$, $k\in\N_0$ and $n\in\N$, we write $(n,j)=(n,j_1,\dots,j_k)$. In particular, $(n,j)=(n,j_1,...,j_k)$
is an individual of generation $|j|=k$. 
When a particle $(n,j)\in \N\times \{1,2\}^k$ reproduces, it disappears and generates particles $(n,j_1,...,j_k,1)$ and $(n,j_1,...,j_k,2)$. For  a subset $A=I\times J $ with  $I\subset \N$ and $J\subset \mathcal J$, we write 
$(\eta_t^{N,A})$ for the process restricted to the subset of particles labelled by elements of $A$, that is, at time $t$ we look at $\eta_t^N$ and ignore all the particles with labels not belonging to $A$. Since 
$$
\eta_t^{N}
	=\sum_{j \in \mathcal J} \eta_t^{N,\N\times \{j\}}
$$
we have
$$
\Sigma_T^{N}
	:=\sup_{0\leq t\leq T} \frac{1}{N^n}\sum_{x\in V}\eta^{N}_t(x)  H(x/N)
	\leq \sum_{j \in \mathcal J} \Sigma_T^{N,j}
$$
where 
$$
\Sigma_T^{N,j}
	:=\sup_{0\leq t\leq T} \frac{1}{N^n}\sum_{x\in V}\eta^{N,\N\times\{j\}}_t(x)  H(x/N)\,.
$$
Using that $\sum_{j \in \mathcal J} 4^{-\vert j\vert} =\sum_{k\geq 0} 2^k4^{-k}= 2$, we can bound
\begin{align}\label{puffi}
P^\omega_{\nu_{0,\rho}}(\Sigma_T^{N}\geq A)
	\leq P^\omega_{\nu_{0,\rho}}\left(\cup_{j \in \mathcal J} \{\Sigma_T^{N,j}\geq  4^{-\vert j\vert} A/2 \}\right)
	\leq  \sum_{j \in \mathcal J} P^\omega_{\nu_{0,\rho}}\left(\Sigma_T^{N,j}\geq  4^{-\vert j\vert} A /2\right)\,.
\end{align}

The key point is now to see that the process $(\eta_t^{N,\N\times \{j\}})$ can be dominated by another process $(Y_{t}^{N,j})$,  obtained by a percolation procedure on the initial distribution of particles. More precisely, $Y_{0}^{N,j}$ is obtained from  $\eta_0^{N}$ as follows: for $\ell\in\N$, the particle with label $\ell$ in $\eta_0^N$ is kept in $Y_{0}^{N,j}$ only if particle $(\ell,j)$ is born before time $T$ for the process $(\eta_t^{N})$. Notice that this happens with probability
\begin{align}\label{lametta}
p_j
	=\P(\text{Poisson} (bT)\geq \vert j\vert)
	={\rm e}^{-bT}\sum_{k\geq \vert j\vert } \frac{(bT)^k}{k !}
\end{align}
since, along each lineage, birth events follow a Poisson process with intensity $b$. If present in $Y_{0}^{N,j}$, then, particle $\ell$ evolves in the process $(Y_{t}^{N,j})$ by following the trajectory of $(\ell,j)$ and its ancestors in $(\eta_0^N)$; once $(\ell,j)$ has disappeared in $(\eta_0^N)$, the particle continues to evolve following the trajectory of any lineage of descendants of $(\ell,j)$.
From this coupling, it is clear that, for all $t\in[0,T]$ and $j\in\mathcal J$,
\begin{align*}
\eta_t^{N,\N\times \{j\}}\leq Y_{t}^{N,j}
\end{align*}
and we have obtained 
\begin{align*}
P^\omega_{\nu_{0,\rho}}\left(\Sigma_T^{N,j}\geq  4^{-\vert j\vert} A /2\right)
	\leq P^\omega_{\nu_{0,\rho}}\Big(\sup_{0\leq t\leq T} \frac{1}{N^n}\sum_{x\in V}Y_t^{N,j}(x) H(x/N)>4^{-\vert j \vert} A/2\Big)\,.
\end{align*}
At this point, we can use Lemma \ref{KVconservative} for the process $(\tilde Y_{t}^{N,j})_{t\in[0,T]}$ on $V/N$, with $\tilde Y_{t}^{N,j}(x/N)=Y_{t}^{N,j}(x)$ for all $x\in V$ and $t\in[0,T]$, since for this process the present particles just perform independent random walks on $V/N$ generated by $L^N$. We notice furthermore that the initial particles of process $(Y_{t}^{N,j})$ have distribution $\nu_{0,\rho p_j}$, cfr.~\eqref{lametta}.
Hence
\begin{align*}
P^\omega_{\nu_{0,\rho}}\Big(\sup_{0\leq t\leq T} \frac{1}{N^n}\sum_{x\in V_N}Y_t^{N,j}(x) H(x/N)>4^{-\vert j \vert} A/2\Big)
	\leq 4^{\vert j \vert} 2A^{-1} c(\rho p_j,T)\vertiii{H}_N
\end{align*}
where the function $c(\cdot,\cdot)$ is the same appearing in Lemma \ref{KVconservative}. 
Going back to \eqref{puffi} we have obtained
\begin{align}\label{conferenza}
P^\omega_{\nu_{0,\rho}}(\Sigma_T^{N}\geq A)
	\leq \sum_{j \in \mathcal J} 4^{\vert j \vert} 2A^{-1} c(\rho p_j,T)\vertiii{H}_N
	\leq \bar c(\rho,T)A^{-1}\sum_{j \in \mathcal J} 4^{\vert j \vert}p_j^{1/2}.
\end{align}
Recall that if $X\sim Poisson(\lambda)$ one has the bound $\P(X\geq t)\leq \e^{-\lambda}({\rm e}\lambda/t)^t$ for all $t>\lambda$ (see for example \cite[Exercise 2.3.3]{V18}). Using also that the number of $j$'s of length $\ell$ is $2^\ell$, we set $\bar\ell=\lceil 81\e bT\rceil$ and compute
\begin{align*}
\sum_{j \in \mathcal J} 4^{\vert j \vert}p_j^{1/2}
	\leq \sum_{\ell=0}^{\bar\ell} 8^\ell + \sum_{\ell=\bar\ell+1}^\infty 8^\ell \Big(\frac{\e bT}{\ell}\Big)^{\ell/2} 
	\leq c \e^{c_2 bT}
\end{align*}
with $c,c_2>0$ absolute constants, which together with \eqref{conferenza} yields the result of the lemma.
\end{proof}

\section{Proof of Theorem \ref{maintheorem}}\label{proofofmaintheorem}

As in Section \ref{sectionexistence}, through the whole section we fix some realization of the underlying graph $\omega\in\Omega$ sampled according to measure $\P$. All the processes in what follows will evolve under measure $P^\omega$, and all the claims have to be intended to be true $\P$--almost surely.

\subsection{An $L^2$ martingale}\label{l2martingale}
In this section we will pave the way for the proof of tightness of the sequence of process $\left((\langle\pi_t^N,G\rangle)_{t\in[0,T]}\right)_N$ and identification of the limit. 
For $G\in\mathcal C_c^\infty(\R^n)$ let us define the process 
\begin{align}\label{dynkin}
	M_t^N=M_t^N(G_N^\l)
	:=\langle\pi_t^N,G_N^{\lambda}\rangle-\langle\pi_0^N,\gnl\rangle -\int_0^t \langle
	\pi_s^N,  L^N G_N^{\lambda}+bG_N^{\lambda} \rangle {\rm d}s \,.
\end{align}
By Lemma \ref{homogenization} and Lemma \ref{stima}, we know that $M^N$ is almost surely well defined when starting from some $\eta_0^N$ satisfying \eqref{domiziano2}. We aim at proving the following result:
\begin{lem}\label{nomartingale}
	Consider a sequence of initial configurations $(\eta_0^N)_{N\in\N}$ satisfying the Domination $\&$ Convergence Assumption.
	For all $\varepsilon>0$ and for all $G\in\mathcal C_c^\infty(\R^n)$ it holds
	\begin{align*}
	\lim_{N\to\infty}  P^\omega\Big(\sup_{0\leq t\leq T}\big|M_t^N\big|\geq \varepsilon\Big)
	=0\,.
	\end{align*}
\end{lem}
In fact, we will not only show Lemma \ref{nomartingale}, but also that $M^N$ is a square integrable martingale which converges in $L^2$ to $0$ and obtain a speed of convergence, see next Lemma \ref{burntends}. To do so, we will use a truncation argument and exploit the results already obtained in Section \ref{finiteparticles} while constructing the process. More precisely, for any $a\in\N$,
consider the process $(\pi^{N,a}_t)_{t\in[0,T]}$ (and the corresponding $(\eta^{N,a}_t)_{t\in[0,T]}$) where the initial configuration of particles is truncated outside the box $[-a,a]^n$, that is, only the particles in the finite set $V\cap [-a,a]^n$ are retained for the initial configuration and all the others are deleted.
By \eqref{siegeauto}, we know that all the particles of the process $\pi^{N,a}$ a.s.~remain in a finite box during time interval $[0,T]$ (recall that $\wpi_t^{B_N}$ was the process restricted to a box of size $B_N$ and that for finite initial conditions $\pi^{N}_t$ was obtained as the restriction to the second coordinate of the limit for $N\to\infty$ of $\wpi_t^{B_N}$, cfr.~Corollary \ref{giardinetti} and \eqref{rabarbaro}).
It follows that the number of births and of jumps is a.s.~finite in a finite time interval and therefore the following equation holds for any locally bounded function $H$: 
\begin{align}\label{napoleon}
\langle\pi_t^{N,a},H\rangle
	=\langle\pi_0^{N,a},H\rangle &+\frac{1}{N^{n}}\int_0^{t} \int_{\R^+}\sum_{x,y\in V}\1{u\leq \eta_s^{N,a}(x)N^2r(x,y)}\big(H(y/N)-H(x/N) \big) \mathcal N^{x,y}({\rm d}s,{\rm d}u) \nonumber\\
	&  +\frac{1}{N^n}\int_0^{t}  \int_{\R^+} \sum_{x\in V}\1{u\leq b\eta_s^{N,a}(x)} \,H(x/N) \, \mathcal Q^{x}({\rm d}s,{\rm d}u)\,.
\end{align}
Notice that we have adopted here a slightly different description of the process for convenience. The underlying Poisson point processes are indexed by sites and not by individuals as before. That is, measures $\mathcal N^{x,y}$ and $\mathcal Q^{x}$ with intensity ${\rm d} s \,{\rm d}u$ on $\R_+^2$ are replacing the previous $\mathcal N_i^{X_s^i,y}$ and $\mathcal N_i^b$.
Equation \eqref{napoleon} can be rewritten as
\begin{align}
	\label{decinit}
	\langle\pi_t^{N,a},H\rangle&=\langle\pi_0^{N,a},H\rangle +\int_0^t \langle\pi_s^{N,a} , L^NH+bH\rangle \, {\rm d}s+M^{N,a}_t(H),
\end{align}
where $M^{N,a}(H)$ is defined by
\begin{align*}
	M^{N,a}_t(H)=& \frac{1}{N^n}\int_0^{t}  \int_{\R^+} \sum_{x,y\in V}\1{u\leq \eta_s^{N,a}(x)N^2r(x,y)}\big(H(y/N)-H(x/N) \big) \widetilde{\mathcal N}^{x,y}({\rm d}s,{\rm d}u) \\
	&  +\frac{1}{N^n} \int_0^{t}  \int_{\R^+} \sum_{x\in V}\1{u\leq b\eta_s^{N,a}(x)} \,H(x/N) \, \widetilde{{\mathcal Q}}^{x}({\rm d}s,{\rm d}u),
\end{align*}
and $\widetilde{\mathcal N}^{x,y}$ and $\widetilde{\mathcal Q}^{x}$ are the compensated 
measures of $\mathcal N^{x,y}$ and ${\mathcal Q}^{x}$.

We turn our attention to $H=G_N^\l$. On the one hand, $\langle\pi_t^{N,a},\gnl\rangle$ increases a.s.~as $a\rightarrow \infty$ to $\langle\pi_t^{N},\gnl\rangle$, which is a.s.~finite (using for example \eqref{allesliebe} and Lemma \ref{homogenization}).
On the other hand,  the fact that $L^N\gnl=\lambda \gnl-H_N$ (cfr.~\eqref{glambda}) and \eqref{allesliebe} ensure that
$$\int_0^t \langle\pi_s^{N} , \vert L^N\gnl \vert+b\gnl \rangle \, ds<\infty \quad \text{a.s.}$$
and it follows by bounded convergence that a.s.
$$
\lim_{a\rightarrow \infty} \int_0^t \langle\pi_s^{N,a} , L^N\gnl +b\gnl \rangle \, {\rm d}s
	=\int_0^t \langle\pi_s^{N} , L^N\gnl +b\gnl \rangle \, {\rm d}s\,.
$$
We obtain from \eqref{decinit} that for any $t\geq 0$, $M^{N,a}_t=M^{N,a}_t(\gnl)$ converges a.s.~as $a\to\infty$  to $M_t^N$, which is given by \eqref{dynkin} and is a.s.~finite.

To wrap up, we have defined a  c\`adl\`ag process $(M_t^N)_{t\in[0,T]}$ satisfying identity \eqref{dynkin} 
and such that, for any $t$, $M^N_t$ is the a.s.~limit of $M^{N,a}_t$, defined as an integral against compensated jump measures.
Let us  check now that these processes are also  square integrable martingale  and that they  tend to $0$ in $L^2$ and probability as $N\to\infty$. This in particular implies Lemma \ref{nomartingale}.



\begin{lem}\label{burntends}
	For any $N\geq 1$ and $a>0$, $M^{N,a}$ and $M^N$
	are c\`adl\`ag  square integrable martingales and, for any $T>0$,
	$$
	E^\omega\Big[\sup_{t\leq T} \big(M^{N,a}_t-M^{N}_t\big)^2\Big]\stackrel{a\rightarrow \infty}{\longrightarrow}0\,.
	$$
	Furthermore, for any $a>0$ and $N\geq 1$,
	\begin{align}\label{visciole}
	E^\omega\Big[\sup_{t\leq T} (M^{N,a}_t)^2\Big]+E^\omega\Big[\sup_{t\leq T}(M^{N}_t)^2\Big]
		\leq \frac{C_T}{N^n}
	\end{align}
	for some constant $C_T$ which only depends on $T$.
\end{lem}
\begin{proof}
	We first prove that $M^{N,a}=M^{N,a}(\gnl)$ is a square integrable martingale. Its quadratic variation is
	\begin{align*}
		\langle M^{N,a}\rangle_{t}=&\frac{N^2}{N^{2n}}\int_0^{t} \sum_{x,y\in V}\eta_s^{N,a}(x)r(x,y)\left(G_N^{\lambda}(y/N)-G_N^{\lambda}(x/N) \right)^2{\rm d}s \\
		&  + \frac{1}{N^{2n}}\int_0^{t}\sum_{x\in V}b\eta_s^{N,a}(x)G_N^{\lambda}(x/N)^2\, {\rm d}s\,.
	\end{align*}
	Since 	$E^\omega[\eta_s^{N,a}(x)]\leq E^\omega[\eta_s^N(x)]\leq C\e^{bs}$ (cfr.~equation \eqref{manzoni} and recall that $\eta_s$ is the projection on alive $a$--particles for $Z_s$) we  get 
	\begin{align}\label{penelope}
		E^\omega\left[\langle M^{N,a}\rangle_{t}\right]
			&\leq C'e^{bt}\Big(\frac{N^2}{N^{2n}} \sum_{x,y\in V}r(x,y)\left(G_N^{\lambda}(y/N)-G_N^{\lambda}(x/N) \right)^2 
		+\frac{1}{N^{2n}}\sum_{x\in V}bG_N^{\lambda}(x/N)^2 \Big)\,.
	\end{align}
	Rewriting
	\begin{align*}
		&r(x,y)\left(G_N^{\lambda}(y/N)-G_N^{\lambda}(x/N) \right)^2\\
		&\quad =-r(x,y)G_N^{\lambda}(x/N)\left(G_N^{\lambda}(y/N)-G_N^{\lambda}(x/N) \right)-r(y,x)
		G_N^{\lambda}(y/N)\left(G_N^{\lambda}(x/N)-G_N^{\lambda}(y/N) \right)
	\end{align*}
	we obtain
	\begin{align}\label{bornvarqu}
		E^\omega\left[\langle M^{N,a}\rangle_{t}\right]
			& \leq \frac{ C''\e^{2bt}}{N^n}\big(( G_N^{\lambda},-L^NG_N^{\lambda})_{\mu_N}
		+ \|  G_N^{\lambda} \|_{L^2(\mu_N)}^2 \big)
	\end{align}
	which is finite by \eqref{gnlinner} and \eqref{gnll1}. It follows that $M^{N,a}$ is a square integrable martingale
	and using Doob's inequality we also obtain the relative $L^2$  bound appearing in \eqref{visciole}.
	 
	We prove now by Cauchy criterion  that $M^{N,a}$ converges to some right-continuous square integrable martingale, since the  space of $L^2$ right-continuous martingales is complete (see e.g.~\cite[Lemma 2.1]{IW89}). By uniqueness, this limit  will then have to be 
	$M^N$. Notice that \eqref{visciole} will automatically follow, since the $L^2$ bound for $M^N$ can be derived from that of $M^{N,a}$ taking the limit.
	More precisely let $a<a'$. Then
	\begin{align*}
		M^{N,a'}_t-M^{N,a}_t
			=& \frac{1}{N^{n}}\int_0^{t}  \int_{\R^+} \sum_{x,y\in V}\1{ \eta_s^{N,a}< u\leq \eta_s^{N,a'}(x)N^2r(x,y)}\left(G(y/N)-G(x/N) \right) \widetilde{\mathcal N}^{x,y}({\rm d}s,{\rm d}u) \\
			&+\frac{1}{N^n}\int_0^{t}  \int_{\R^+}\sum_{x\in V}\1{b\eta_s^{N,a}(x) < u\leq b\eta_s^{N,a'}(x)} \,G(x/N) \, \widetilde{{\mathcal Q}}^{x}({\rm d}s,{\rm d}u),
	\end{align*}
	and 
	\begin{align*}
		\langle M^{N,a'}-M^{N,a}\rangle_{t}
			=&\frac{N^2}{N^{2n}}\int_0^{t} \sum_{x,y\in V}\big(\eta_s^{N,a'}(x)-\eta_s^{N,a}(x)\big)\,r(x,y)\left(G_N^{\lambda}(y/N)-G_N^{\lambda}(x/N) \right)^2{\rm d}s \\
		&+\frac{1}{N^{2n}}\int_0^{t}\sum_{x\in V}b\big(\eta_s^{N,a'}(x)-\eta_s^{N,a}(x)\big)G_N^{\lambda}(x/N)^2 {\rm d}s\,.
	\end{align*}
	As $M^{N,a'}-M^{N,a}$ is a square integrable martingale and by monotonicity with respect to $a'$, 
	\begin{align*}
		E^\omega\Big[&\sup_{t\leq T} (M^{N,a'}_t-M^{N,a}_t)^2\Big]\\
			\leq& \frac{4N^2}{N^{2n}}\int_0^{t} \sum_{x,y\in V}E^\omega\big[\eta_s^{N}(x)-\eta_s^{N,a}(x)\big]r(x,y)\left(G_N^{\lambda}(y/N)-G_N^{\lambda}(x/N) \right)^2{\rm d}s \\
		& + \frac{4}{N^{2n}}\int_0^{t}\sum_{x\in V}bE^\omega\big[\eta_s^{N}(x)-\eta_s^{N,a}(x)\big]G_N^{\lambda}(x/N)^2 {\rm d}s
	\end{align*}
	We can apply bounded convergence, see \eqref{penelope} and \eqref{bornvarqu}  to get that the right hand side goes to $0$ as $a$ goes to infinity. 
\end{proof}

\subsection{Proof of tightness}\label{proofoftightness}

Having proved in the previous section that the process $M^N$ defined in \eqref{dynkin} is an $L^2$ martingale makes now the proof of tightness quite straight-forward by using Aldous criterion.

\begin{lem} \label{cittina}
	Consider a sequence of initial configurations $(\eta_0^N)_{N\in\N}$ satisfying the Domination $\&$ Convergence Assumption.
	The sequence of processes $\{(\pi_t^N)_{t\in[0,T]}\}_{N\in\N}$ is tight in $D([0, T ],\mathcal M)$.
\end{lem}

Since $\mathcal M$  is separable and the vague topology in $\mathcal M$ is metrizable, in order to prove tightness of $\pi^N_t$ in $D([0, T ],\mathcal M)$ it is enough to show tightness of $\langle\pi^N_t,G\rangle$ in  $D([0, T ],\R)$ for $G$ in a dense subset of $\mathcal C(\R^n)$, the set of continuous functions in $\R^n$. We consider functions $G\in\mathcal C_c(\R^n)$ that are continuous with compact support. The proof of this fact can be obtained by following \cite[Proposition 1.7, Chapter IV]{KL98} or \cite{MR93}. Furthermore, the tightness of $\langle\pi^N_t,G\rangle$ can be deducted by that of $\langle\pi^N_t,\gnl\rangle$:

%

\begin{lem}\label{tightness}
	For any $G\in\mathcal C^\infty_c(\R^n)$ and for $\P-a.a.$ $\omega$ the following holds.
	Consider a sequence of initial configurations $(\eta_0^N)_{N\in\N}$ satisfying the Domination $\&$ Convergence Assumption.
	Then
	the process $(\langle\pi_t^N,\gnl\rangle)_{t\in[0,T]}$ is well defined for all $N\in\N$. The tightness of $\{(\langle\pi_t^N,\gnl\rangle)_{t\in[0,T]}\}_{N\in\N}$ in $D([0, T ],\R)$ implies the tightness of $\{(\langle\pi_t^N,G\rangle)_{t\in[0,T]}\}_{N\in\N}$ in $D([0, T ],\R)$ and therefore the tightness of $\{(\pi_t^N)_{t\in[0,T]}\}_{N\in\N}$ in $D([0, T ],\mathcal M)$\,.
\end{lem}

\begin{proof}
Notice that the combination of Lemma \ref{homogenization} and of Lemma \ref{stima} guarantees that the process $(\langle\pi_t^N,\gnl\rangle)_{t\in[0,T]}$ is well defined. 
Furthermore,   for each $\varepsilon>0$,  	
\begin{align}\label{tighttight}
	\lim_{N\to\infty}P^\omega\Big(\sup_{0\leq t\leq T}\big| \langle\pi_t^N,\gnl\rangle-\langle\pi_t^N,G\rangle \big|\geq \varepsilon  \Big)
		=0\,.
	\end{align}
This fact can be shown by using Lemma \ref{stima} to bound the probability in  \eqref{tighttight} by
\begin{align*}
\varepsilon^{-1} c_1{\rm e}^{ c_2bT} \sqrt{\|\gnl-G\|^2_{L^1(\mu_N)}+N^{-n}\|\gnl-G\|_{L^2(\mu_N)}
\|L^N(\gnl-G)\|_{L^2(\mu_N)}}
\end{align*}
and then applying Lemma \ref{homogenization} and Lemma \ref{tiramisu} to see that this converges to $0$ as $N\to\infty$.
\end{proof}

\begin{proof}[Proof of Lemma \ref{cittina}] By Lemma \ref{tightness} we will just have to prove tightness of $(\langle\pi_t^N,\gnl\rangle)_{t\in[0,T]}$. We will use Aldous criterion, see for example \cite[Section 4: Proposition 1.2 and Proposition 1.6]{KL98}.
Let us work with the set $\mathcal T_N(\theta)$ of couples $(\tau,h)$
such that $\tau$ is
 a stopping time for the process $(\eta_t^N)_{t\in[0,T]}$ and $h\in\R$ is such that   $0\leq h \leq \theta $ and  $\tau+h\leq T$.  
Using identity   \eqref{dynkin} we have
\begin{align}\label{dynkinn}
	\langle\pi_{\tau+h}^N-\pi_{\tau}^N,\gnl\rangle\
		&=\int_{\tau}^{\tau+h} \langle \pi_s^N,  L^N G_N^{\lambda}+ b G_N^{\lambda}\rangle {\rm d}s+ M_{\tau,\tau+h}^N(G),
\end{align}
where $M_{\tau,\tau+h}^N(G)=M_{\tau+h}^N(G)-M_{\tau}^N(G)$. 

Equation \eqref{definitiongnl} allows us to write $L^N G_N^{\lambda}=\lambda (G^{\lambda}_N-G)+\sigma^2\Delta G$, so that
 \ben
 &&\left\vert\int_{\tau}^{\tau+h}  \langle \pi_s^N,  L^N G_N^{\lambda}\rangle\, {\rm d}s\right\vert \leq 
 h\sup_{0\leq s\leq T}  \langle \pi_s^N ,  \lambda \vert G^{\lambda}_N-G \vert  +\vert\sigma^2\Delta G\vert \rangle.
 \een
We want to use Lemma \ref{stima} to this quantity and then take the limit $\theta\downarrow 0$. Using that $\|\gnl-G\|_{L^1(\mu_N)}$, $\|\gnl-G\|_{L^2(\mu_N)}$ and $\|L^N \gnl\|_{L^2(\mu_N)}$ are bounded by Lemma \ref{homogenization}, using that $G$ is compactly supported and using also Lemma \ref{tiramisu}, we obtain
$$
\lim_{\theta\downarrow 0} \sup_{\substack{(\tau,h) \in \mathcal T_N(\theta) \\ N\geq 1} }P^\omega\left( \left\vert\int_{\tau}^{\tau+h}  \langle \pi_s^N,  L^N G_N^{\lambda}\rangle\, {\rm d}s\right\vert \geq \delta\right)
= 0 
$$
for any $\delta>0$.
Similarly we also see that
$$
	\lim_{\theta\downarrow 0} \sup_{\substack{(\tau,h) \in \mathcal T_N(\theta)\\ N\geq 1}} P^\omega\left(\left\vert \int_{\tau}^{\tau+h} \langle \pi_s^N, bG_N^{\lambda}\rangle \, {\rm d}s\right\vert \geq \delta\right)
= 0\,.
$$
Finally we also know that the contribution of $M_{\tau,\tau+h}^N(G)$ is negligible by Lemma \ref{nomartingale}.
All in all equation \eqref{dynkinn} combined with these estimates   yields
$$\lim_{\delta\downarrow 0}  \limsup_{N\rightarrow \infty} \sup_{(\tau,h) \in \mathcal T_N(\theta)} P^\omega\left (\vert \langle\pi_{\tau+h}^N,\gnl\rangle-\langle \pi_{\tau}^N,\gnl\rangle\vert \geq 3\delta\right)=0 $$
for any $\delta>0$, which ends the proof.
\end{proof}

\subsection{Proof of   identification and convergence }\label{identificare}
We are finally ready to identify the limiting points
of  $(\pi_t^N)_{t\in[0,T]}$ as deterministic measure-valued processes with density. We will show that this limit can be  
characterized as the weak solution of \eqref{reactiondiffusion}. 
\begin{lem}\label{gvsg}
	For $\P-a.a.$~$\omega$ the following holds.
	Consider a sequence of initial configurations $(\eta_0^N)_{N\in\N}$ satisfying the Domination $\&$ Convergence Assumption.
	For all $\varepsilon>0$ and $G\in \mathcal C(\R^n)$, it holds
	\begin{align*}
	\lim_{N\to\infty}P^\omega\Big(\sup_{0\leq t\leq T}\Big|\int_0^t \,\langle \pi_s^N,	G_N^\l-G \rangle	\,{\rm d}s\Big|>\varepsilon\Big)
		=0\,.
	\end{align*}
\end{lem}
\begin{proof} As a direct consequence of Lemma \ref{stima}  it holds
\begin{align*}
P^\omega\Big(\sup_{0\leq t\leq T}\Big|\int_0^t \,\langle \pi_s^N,
	G_N^\l-G \rangle\, {\rm d}s\Big|> \varepsilon\Big)
	\leq \varepsilon^{-1}T c_1 e^{c_2bT}\vert\kern-0.25ex\vert\kern-0.25ex\vert \gnl-G 
	\vert\kern-0.25ex\vert\kern-0.25ex\vert_N
\end{align*}
with $\vertiii{\cdot}$ defined in \eqref{babu}. As in the proof of Lemma \ref{tightness}, one concludes by observing that $\vert\kern-0.25ex\vert\kern-0.25ex\vert \gnl-G 
\vert\kern-0.25ex\vert\kern-0.25ex\vert_N$ goes to $0$ as $N\to\infty$.
\end{proof}

Thanks to 
Lemma \ref{tightness}, we can consider now a limiting value of of  $(\pi_t^N)_{t\in[0,T]}$, call it $(\pi_t)_{t\in[0,T]}$, in the space $D([0, T ],\mathcal M)$, where $\mathcal M=\mathcal M(\R^n) $ is as usual endowed with the vague topology.  
Noticing that \eqref{glambda} is equivalent to 
\begin{align*}
L^NG_N^\l=\l(G_N^\l-G)+\sigma^2\Delta  G\,,
\end{align*}
we can use representation \eqref{dynkin} and put together the results of Lemma \ref{nomartingale} and Lemma  \ref{gvsg}
and {\eqref{tighttight}}
  to infer that for any $\varepsilon>0$
\begin{align}\label{finally}
\lim_{N\to\infty}
	P^\omega\Big(\sup_{t\in[0,T]}\Big|\langle\pi_t^N,G\rangle-\langle\pi_0^N,G\rangle-\int_0^t \langle \pi_s^N, \sigma^2\Delta G + bG \rangle \,{\rm d}s\Big|>\varepsilon\Big)
	=0\,.
\end{align}
As $\langle \pi_t^N,G\rangle$ converges to $\langle  \pi_t,G\rangle$ along a subsequence, we get, for any $G\in \mathcal C_c^{\infty}(\R^n)$,
\begin{equation}
\label{eqlimite}
\langle  \pi_t,G\rangle=\langle  \pi_0,G\rangle +\int_0^t \langle \pi_s  ,\sigma^2\Delta G + bG \rangle\,{\rm d}s
\end{equation}
$\P$--almost surely.
We use now classical techniques to check that the solutions have a density and prove uniqueness of the limiting problem. The result may be classical, even if the fact that our domain is non bounded or our weak formulation make that we do not know the appropriate reference. For   convenience of the reader, we provide the proof.

\begin{lem}\label{halls} 
i) For $\P$-a.a.~$\omega$ the following holds.
Consider a sequence of initial configurations $(\eta_0^N)_{N\in\N}$ satisfying the Domination $\&$ Convergence Assumption for some $\rho_0$. Consider a limiting point $\pi$ of $\pi^N$ in  the space $D([0, T ],\mathcal M)$. Then $\pi$ satisfies Equation \eqref{eqlimite}, with all terms well defined, for any $G$ in 
	$$
	\mathcal G=\{G\in \mathcal C^{\infty}(\R^n)\cap L^1({{\rm d}x}) :  \sigma^2\Delta G  \in L^1({{\rm d}x}) \}.$$
ii)  Equation \eqref{eqlimite} has a unique weak solution in the space cadlag  positive Radon measure valued  function. Besides this solution $\pi$ has a bounded density :
	$\pi_t({\rm d}x)=\rho_t(x){\rm d}x$ a.s., for a.e.~$t\geq 0$,
	and $\rho_\cdot \in L^{\infty}([0,T]\times \R^n, \R_+)$.
\end{lem}
The second part of the statement ensures that the limiting processes $\pi$ are deterministic. For the proof of the existence of a density, we exploit our results coming from Kipnis Varadhan estimates, but one may also invoke results focusing on the solution of \eqref{eqlimite}.
\begin{proof}
Let us prove $(i)$ and consider $G\in\mathcal G$ non-negative. Approximate $G$ with non-negative functions $G_k\in \mathcal C^\infty_c(\R^n)$ such that $G_k(x)=G(x)$ for $x\in B_k$, with $B_k$ the ball of radius $k$ centered in the origin, and $G_k(x)\leq G(x)$ for $x\notin B_k$. Then, as $k\to\infty$, dominated convergence theorem yields that $\langle  \pi_t,G_k\rangle \to \langle  \pi_t,G\rangle$, $\langle  \pi_0,G_k\rangle\to \langle  \pi_0,G\rangle$ and $\int_0^t \langle \pi_s  , bG_k \rangle\,{\rm d}s\to \int_0^t \langle \pi_s  , bG \rangle\,{\rm d}s$.
It follows that also $G$ satisfies \eqref{eqlimite}, the r.h.s.~being well defined and integrable by part $(i)$.

 Let us move to $(ii)$. 
Let us move to $(ii)$.  For the purpose, we consider the subspace of signed  Radon measures on $\R^n$ associated with 
the norm   
$$ 
\| \mu  \|_{TV}=\sup_{f \in \mathcal G, \, \| f \|_{L^1({\rm d}x)}\leq 1} \vert \langle\mu ,f\rangle \vert\,.
$$
We also introduce the space $\mathcal G_T$ of functions  $G  : [0,T]\times \R^n\rightarrow \R$ that verify the following properties:
$G(t,.)\in \mathcal G$ for any $t\in [0,T]$; $\sup_{t\in [0,T]} \left\{ \| G(t,.)\|_{L^1({\rm d}x)}+  \| \sigma^2\Delta G (t,.)\|_{L^1({\rm d}x)}\right\}<\infty$; $s\rightarrow G(s,x)$ is differentiable for any 
$x\in \R^n$ and $\sup_{s\in[0, T]} \vert \partial_s G(s,.) \vert \in L^1({\rm d}x)$. 
By approximating $G(\cdot,\cdot)\in \mathcal G_T$ with a function which is piece-wise linear in time and proceeding as in \cite[Section 3]{GJ08}, we observe  that  \eqref{eqlimite}  can be extended to functions in  $\mathcal G_T$. This is the weak form of \eqref{reactiondiffusion}.

Finally, to get uniqueness, consider $G^T(t,x)=P_{T-t}\varphi(x)$ with $\varphi \in \mathcal G$  where $P_tf(x)=\E[f(x+B_t)]$ is the semigroup of the Brownian motion with generator $\sigma^2\Delta$. 
Then $\partial_s G^T(s,\cdot)=-\sigma^2\Delta G^T (s,\cdot)$ and $G^T\in \mathcal G_T$ 
and 
$$\langle  \pi_t,G^T(t,.)\rangle=\langle  \pi_0,G(0,.)\rangle +\int_0^t  \langle \pi_s , bG^T(s,.) \rangle\,{\rm d}s.$$
This implies, for two solutions $\pi^1$ and $\pi^2$ of Equation \eqref{eqlimite} with 
same initial values $\pi_0$, 
$$\langle  \pi^1_T-\pi^2_T,\varphi \rangle=\int_0^T  \langle \pi^1_s-\pi^2_s,  bG^T(s,\cdot) \rangle\,{\rm d}s$$
Adding that 
$ \|G^T(t,.)\|_{L^1({\rm d}x)}=  \|\varphi \|_{L^1({\rm d}x)}$, we obtain 
$$
\| \pi^1_T-\pi^2_T  \|_{TV} \leq b\int_0^T \| \pi^1_s-\pi^2_s  \|_{TV} \, {\rm d}s \,.
$$
Gronwall lemma yields $\pi^1=\pi^2$ and uniqueness of the solution $\eqref{eqlimite}$ in  the space of positive Radon measure is proved. \\

Finally we prove that this solution admits a density. We can exploit the first part $i)$
and see this solution $\pi$ as the limit of our sequence of processes.
Take $H\in \mathcal C^\infty_c$ to be a nonnegative function on $V/N$ belonging to $L^1(\mu_N)$ and $L^2(\mu_N)$.  Letting $N\to\infty$ in equation \eqref{allesliebe}, we invoke Portemanteaux theorem and the fact that the supremum over $[0,T]$ of $\langle \, \cdot\,, H \rangle$ is a continuous functional on $D([0,T],\mathcal M(\R^n))$ to obtain 
\begin{align}\label{topo}
	P^\omega_{\rho_0}\Big(\sup_{0\leq t\leq T}\, \langle  \pi_t, H \rangle \, >A\Big)
	\leq A^{-1} c(T)\|H\|_{L^1({\rm d}x)}
\end{align}
with $c(T)=c_1\e^{c_2bT}$. On the r.h.s.~we have used the fact that $N^{-n} \, \| L^NH  \|_{L^2(\mu_N)}$ converges a.s.~to $0$ using Lemma  \ref{tiramisu} and that $\|H\|_{L^1(\mu_N)}$ tends a.s.~to $\|H\|_{L^1({\rm d}x)}$. For a given $\omega$ the process $\pi_t$ must be deterministic, since it verifies \eqref{eqlimite}. Hence, by \eqref{topo}, we have obtained that
$$
\sup_{0\leq t\leq T}\, \langle  \pi_t, H \rangle 
	\leq c(T)\|H\|_{L^1({\rm d}x)}\,.
$$
for any $H\in \mathcal C^\infty_c$. By  approximation,
this identity can be extended to any non-negative $H\in L^1({\rm d}x)$. 
As a consequence, the process $\pi_t$ is concentrated
on measures absolutely continuous with respect to the Lebesgue measure and we call $\rho_t(x)$ its density at time $t\in[0,T]$ and for any $t\in [0,T]$, $\rho_t\leq C(T)$ a.e. and $\rho \in L^{\infty}([0,T]\times \R^n, \R_+)$.
	
It follows that also $G$ satisfies \eqref{eqlimite}, the r.h.s.~being well defined and integrable by part $(i)$.

\end{proof}

\subsection{Proof of Theorem \ref{maintheorem} with mortality ($d>0$)}\label{verydead}
We want to adapt the proof of Theorem \ref{maintheorem} to the case where particles die at rate $d>0$. 

As mentioned before, the construction of the process $(\eta_t)_{t\in[0,T]}$ is still valid when $d>0$. By following the proof of Section \ref{sectionexistence}
one can prove that  the corresponding measure-valued process satisfies an equivalent of \eqref{colosseo3} that now reads, for $G : \R^n\rightarrow \R$ with compact support,
\begin{align}
\langle \pi_t,G \rangle
	=\langle \pi_0, G\rangle 
	&+\int_0^t \sum_{i\in \mathcal A_s, y \in V}  \big(G(y)  -G(X_s^i)\big) \,\NN_{i}^{X_s^i,y}({\rm d}s) \nonumber\\
	&+\int_0^t \sum_{i\in \mathcal A_s} G(X_s^i)  \,\NN_{i}^{b}({\rm d}s)	-\int_0^t \sum_{i\in \mathcal A_s} G(X_s^i)  \,\NN_{i}^{d}({\rm d}s)	
\end{align}
where $(\NN_i^b)_{i\in\mathcal I}$ is another collection of independent Poisson point measures on $\R_+$ with intensity $d\,{\rm d}t$, independent from the $\NN_{i}^{x,y}$'s and the $\NN_{i}^{b}$'s.

There is no problem in updating the technical tools presented in Section \ref{inputfromhomo} and Section \ref{sectionKV} to the case $d>0$. The homogenization results do not depend on the specific particle dynamics, so the results of Lemma \ref{homogenization} are needed as they are. 
The non-reversible Kipnis-Varadhan estimate Lemma \ref{stima} holds when $d>0$, too. To see that, one can couple the dynamics with $d=0$ and the one with $d>0$ in a way that guarantees that the supremum in equation \eqref{allesliebe} always decreases (for example, one can use the same Poisson processes for generating jumps and births of the particles).

The results of Section \ref{l2martingale} still hold for the process 
\begin{align}\label{dynkin}
M_t^N
	=\langle\pi_t^N,G_N^{\lambda}\rangle-\langle\pi_0^N,\gnl\rangle -\int_0^t \langle
\pi_s^N,  L^N G_N^{\lambda}+(b-d)G_N^{\lambda} \rangle {\rm d}s \,.
\end{align}
In particular, $M^N$ is a square integrable martingale and it satisfies Lemma \ref{nomartingale}. This can be again achieved by the same truncation argument. 
Finally the proofs of tightness, identification and convergence (corresponding to Section \ref{proofoftightness} and Section \ref{identificare}) for $d>0$ follow those of the case $d=0$ as they are fundamentally based on the homogenization results and on the non-reversible Kipnis-Varadhan estimates.

\section{Extension of the results and perspective}\label{randomgraphs}

\subsection{Extension to random graphs}
For the sake of clarity we have made the choice to state and prove our main results for the particle system evolving over the complete graph $\mathcal G=(V,E)$, where $V$ are the points of an homogeneous Poisson point process of intensity $\gamma>0$ on $\R^n$, with $n\geq 2$. In this section we discuss how one can consider a broader class of graphs.  To this end, we will consider graphs $\overline {\mathcal G}=(\overline V,\overline E)$ with $\overline V\subseteq V$ and $\overline E\subseteq E$ obtained under measure $\P$ (eventually enlarging the probability space $\Omega$) by performing a random percolation procedure on the edges of $\mathcal G$.
For a graph $\overline {\mathcal G}$ we say that the particle system evolves on $\overline {\mathcal G}$ when the particles move on the nodes $\overline V$ with transition rates $r(x,y)$ substituted by $\overline r(x,y):=r(x,y)\1{\{x,y\}\in\overline E}$.

\smallskip

The next theorem is the generalization of Theorem \ref{existence}. In this case one can perform {\it any} percolation procedure on the bonds of the complete graph $\mathcal G$. Notice that we include cases where the graph becomes disconnected.
\begin{thm1bis}\label{existencebis}
	For $\P$-a.a.~realizations of the underlying Poisson point process, the following holds. Let $\overline {\mathcal G}=(\overline V,\overline E)$ be any graph such that $\overline V= V$ and $\overline E\subseteq E$. 
	Let $\eta_0$ be a  random variable on  $\N^{\overline V}$ such that $\E[\eta_0(x)]\leq M$ for all $x\in\overline V$, for some $M\in\N$.
	 Then, 
	 for all $T>0$,  
there exists a Markov process $(\eta_t)_{t\in[0,T]}$ with initial value $\eta_0$ and 
	paths in the Skohorod space $\mathcal D([0,T],\N^{\overline V})$ that satisfies the following:
	for functions $G$ compactly supported in $\R^n$, the process $(M^G_t)_{t\geq 0}$ defined by
	$$
	M^G_t=\sum_{x\in \overline V} \eta_t(x)\, G(x)-\sum_{x\in \overline V} \eta_0(x)\, G(x) - \int_0^t  \mathcal L f_G(\eta_s) \,{\rm d}s
	$$
	is a martingale. Here  $f_G:\N^{\overline V}\to \R$ is the function $f_G(\eta)=\sum_{x\in \overline V} G(x)\eta(x)$ and 
	\begin{align}\label{generator}
	\mathcal L f_G(\eta)
	=&\sum_{x,y\in \overline V}\eta(x)r(x,y)\big(G(y)-G(x)\big)+\sum_{x\in \overline V}\eta(x)\big(b- d\big) G(x)\,.
	\end{align}
\end{thm1bis}
\begin{proof}
	The existence of the process on the finite graph as in Section \ref{finitegraph} goes easily through. Also the quantitative estimates of Lemma \ref{capogiro} work with $r(x)$ substituted by the corresponding quantity $\overline r(x)=\sum_{y\in \overline V}\overline r(x,y)$. For the case of a finite number of initial particle, cfr.~Section \ref{finiteparticles}, we notice that Lemma \ref{tyran} is true for the percolated graph since $\overline r(x)\leq r(x)$. This in turn implies the equivalent of the key Proposition \ref{seggiolino} and hence of the equivalent of Corollary \ref{giardinetti}. The existence of the process with an infinite number of initial particles can be checked then by  following the proof of Section \ref{full}.
\end{proof}

\smallskip

Generalizing the results of Theorem \ref{maintheorem} is much more subtle.
In this case one is not authorized to freely percolate the edges of $\mathcal G$. Indeed, our proof crucially relies on the homogenization results discussed in Section \ref{input}. On the other hand, as long as the assumptions that imply homogenization are fulfilled, our machinery continues to work and we obtain the following more general version of Theorem \ref{maintheorem}.


\begin{thm2bis}\label{maintheorembis}
	Consider under $\P$ a random graph $\overline {\mathcal G} =(\overline V,\overline E)$, with $\overline V\subseteq V$ and $\overline E\subseteq E$ 
	verifying conditions (A1),\dots,(A9) of Section \ref{input} for some opportune choices of $\Omega$, $(\theta_g)_{g\in\R^n}$ and $\mu_\omega$. Then the results of Theorem \ref{maintheorem} hold for the particle system evolving on $\overline {\mathcal G}$. In this case, the matrix $\sigma\Delta$ appearing in \eqref{reactiondiffusion} has to be substituted by the matrix $D$ such that, for any $a\in\R^n$,
	\begin{align}\label{variazionale2}
	a\cdot Da = \frac 12\inf_{\psi\in L^\infty(\P_0) }
	\E_{0}\Big[\sum_{y\in \overline V}\bar r(0,y)\big(a\cdot y +\psi(\theta_y\omega)-\psi(\omega)\big)^2\Big]\,.
	\end{align}
\end{thm2bis}

\begin{proof}
	The homogenization results of Lemma \ref{homogenization} are still valid exactly because (A1),\dots, (A9) have been chosen so that they would work. 
	Lemma \ref{tiramisu} works since $\overline V\subseteq V$ and since $\bar r(x,y)\leq r(x,y)$ for all $x,y\in V$. The non-conservative Kipnis-Varadhan estimate of Lemma \ref{stima} carries through to the edge-percolated case since the reversibility of the process without births and deaths is maintained on $\overline {\mathcal G}$. Sections \ref{l2martingale}, \ref{proofoftightness} and \ref{identificare} follow a general strategy that relies on the previous estimates and remains substantially identical for the percolated graph. 
\end{proof}

\begin{rem}\label{degenerato}
	It is possible, in principle, that the diffusion matrix described by \eqref{variazionale2} is degenerate. 
\end{rem}

To give more substance to Theorems \ref{existence}' and \ref{maintheorem}', we exhibit now two well-known models that are obtained from the complete graph $\mathcal G=(V,E)$ via a bond-percolation procedure. For both of them we check that assumptions (A1),\dots,(A9) are in fact satisfied, at least for some range of the parameters.

\smallskip

{\bf Long-range percolation.} Long-range percolation is a well studied random graph model, usually defined on the $\mathbb Z^n$ lattice, see for example \cite{B02, H21} and references therein. Extending its definition to continuous space (as for example in \cite{P91}), we consider again the set of vertices $V$ given by a Poisson point process of intensity $\gamma>0$ and the set of all possible edges $E$. Under $\P$, independently for each $\{x,y\}\in E$ we retain the edge with probability $1-\e^{-\beta\|x-y\|^{-\alpha}}$ and delete it otherwise, with $\alpha,\beta>0$ two parameters of the model. Theorem \ref{existence}'  guarantees that our particle process on this structure is always well-defined. 

To study the hydrodynamic limit as in Theorem \ref{maintheorem}' we need some adjustment. 
The main problem arising from the percolation procedure, in fact, is that the resulting graph might not be connected and (A6), for example, would not stand. To overcome this issue we must consider parameters $\alpha, \beta$ and $\gamma$ that guarantee the existence of a unique infinite component for the graph with $\P$--probability $1$ (see \cite{P91} for results on the existence of an infinite component in continuous long-range percolation). Then we define $\overline V$ to be the set of vertices in the infinite component of the graph and  $\overline E$ to be the set of retained edges connecting points of $\overline V$. 

More precisely, let us use the notation of Section \ref{input}. We let $\Omega=S\times[0,1]^{\N^2}$ where $S$ is the set of locally finite sets of $\R^n$. The probability measure $\P$ on $\Omega$ samples locally finite sets $V(\omega)$ according to the law of a Poisson point process of intensity $\gamma>0$. For each couple of vertices $x\neq y\in V(\omega)$ we also sample under $\P$ independent uniform random variables on $[0,1]$, call them $U_{x,y}=U_{y,x}$. We place an edge $\{x,y\}$ between $x$ and $y$ if and only if $U_{x,y}\geq 1-\e^{-\beta\|x-y\|^{-\alpha}}$. We identify $\overline V(\omega)=\hat\omega$ with the points in the infinite component and let $n_x(\omega)=1$ for $x\in \hat\omega$ and $n_x(\omega)=0$ otherwise. The set of edges of the graph is $\overline E(\omega)=\{\{x,y\}:\,x\neq y\in \overline V(\omega)\}$. The maps $(\theta_g)_{g\in\R^n}$ act on an environment $\omega=\big(\{x_i\}_{i\in\N},\,\{u_{x_j,x_k}\}_{j\neq k\in\N}\big)$ as $\theta_g\omega=\big(\{x_i-g\}_{i\in\N},\,\{u_{x_j-g,x_k-g}\}_{j\neq k\in\N}\big)$. The measure $\mu_\omega$ becomes $\mu_\omega(A)=\sum_{x\in\hat\omega}\delta_x(A)$, so that $n_x(\omega)=1$ if $x$ belongs to the infinite component of the graph and $n_x(\omega)=0$ otherwise. 
The rates are $\overline r(\omega,x,y):=r(x,y)\1{\{x,y\}\in\overline E(\omega)}$.

We show now that with these choices assumptions (A1),\dots,(A8) hold and discuss (A9) after. (A1) comes from the stationarity and ergodicity of the Poisson point process combined with the independence of the percolation procedure. Conditions (A2), (A3) and (A7) are inherited from the underlying Poisson point process. (A4) and (A5) are straightforward. 
(A6) is clear since we have restricted ourselves to the infinite component and every jump within it has positive probability. (A8) can be also deducted from the analogous property for the Poisson point process. The only point left to check is the decay of correlation (A9). Clearly $\E[N_0^2]<\infty$. The fact that $|Cov(N_k,N_{k'})|\leq C_0 |k-k'|^{-1}$ is more delicate, since $N_k$ might depend on the structure of the configuration of the infinite cluster in a far away box $k'+[0,1)^n$. While we believe the bound to be true, we could not find a proper reference in the literature and its proof might be technically involved. One way to avoid this problem, is to consider the parameters $\alpha$ and $\beta$ in the range where {\it{all}} the points of $V(\omega)$ belong to the infinite component $\P$--almost surely.
This happens, for example, when $\alpha<n$. In this case (A9) is clearly verified.

\smallskip

{\bf Scale-free percolation.} Scale-free percolation is an inhomogeneous version of the long-range percolation model. It was originally introduced with nodes placed on the lattice $\mathbb Z^n$ in \cite{DVH13} and then also studied on a Poisson point process in \cite{DW18} and \cite{DS21}. Under measure $\P$, we let $V$ be the realization of a Poisson point process of parameter $\gamma>0$ and assign independently to each vertex $x\in V$ a random weight $W_x\in[1,\infty)$ such that $\P(W_x>w)=w^{-(\tau-1)}L(w)$ for some $\tau>1$ and with $L$ a slowly varying function. Then, independently for each edge of the complete associated graph $\{x,y\}\in E$, we retain the edge in $\overline E$ with probability $1-\e^{-\beta W_xW_y\|x-y\|^{-\alpha}}$ for some $\alpha,\beta>0$ and delete it otherwise. The particle system on the resulting graph is again well defined by Theorem \ref{existence}'.

For Theorem \ref{maintheorem}' to work, we restrict to values of $\alpha,\beta, \gamma$ and $\tau$ such that a unique infinite component $\overline V$ exists with $\P$--probability $1$, see \cite{DW18} for the precise range of parameters. The random graph $\overline {\mathcal G}=(\overline V,\overline E)$ induced on the infinite component equipped with rates $\overline r(x,y)$ verifies (A1),\dots,(A8) with the following choices: $\Omega=S\times[1,\infty)^{\N}\times[0,1]^{\N^2}$, where $S$ is the set of locally finite sets of $\R^n$. The probability measure $\P$ on $\Omega$ samples locally finite sets $V(\omega)$ according to the law of a Poisson point process of intensity $\gamma>0$. For each $x\in V(\omega)$ we sample independently under $\P$ the random weight $W_x$ with the law specified here above. Finally, for each couple of vertices $x\neq y\in V(\omega)$ we also sample under $\P$ independent uniform random variables on $[0,1]$, call them $U_{x,y}=U_{y,x}$. We place an edge $\{x,y\}$ between $x$ and $y$ if and only if $U_{x,y}\geq 1-\e^{-W_xW_y\|x-y\|^{-\alpha}}$. We identify $\overline V(\omega)=\hat\omega$ with the points in the infinite component and let $n_x(\omega)=1$ for $x\in \hat\omega$ and $n_x(\omega)=0$ otherwise. The set of edges of the graph is $\overline E(\omega)=\{\{x,y\}:\,x\neq y\in \overline V(\omega)\}$. The maps $(\theta_g)_{g\in\R^n}$ act on an environment $\omega=\big(\{x_i\}_{i\in\N},\,\{w_{x_\ell}\}_{\ell\in\N},\,\{u_{x_j,x_k}\}_{j\neq k\in\N}\big)$ as $\theta_g\omega=\big(\{x_i-g\}_{i\in\N},\,\{w_{x_\ell-g}\}_{\ell\in\N},\,\{u_{x_j-g,x_k-g}\}_{j\neq k\in\N}\big)$. The measure $\mu_\omega$ becomes $\mu_\omega(A)=\sum_{x\in\hat\omega}\delta_x(A)$, so that $n_x(\omega)=1$ if $x$ belongs to the infinite component of the graph and $n_x(\omega)=0$ otherwise. 
	The rates are $\overline r(\omega,x,y):=r(x,y)\1{\{x,y\}\in\overline E(\omega)}$. All properties (A1),\dots, (A8) can be deducted as for the long-range percolation model. To verify (A9) we have the same problem as before with correlations. To be on the safe side, one can take $\alpha\leq n$ or $\alpha(\tau-1)/n\leq 1$, which guarantees all the points of $V$ to be in the infinite cluster $\P$--almost surely, see \cite{DS21}.

\subsection{Open problems and perspectives}\label{generalizations}
The methods we present in our paper can be easily used for a wider class of models. First of all, we point out that the reason for choosing a Poisson point process as a support is mainly due to the seek of clarity rather than to technical obstacles of dealing with more general settings. It should be possible to replace the Poisson point process with any sufficiently regular simple point process satisfying the nine conditions of \cite{F22}, see Section \ref{input}. Likewise, the transition rates $r(x,y)$ do not have to assume the exact form $\e^{-\|x-y\|}$: any rates decaying sufficiently fast (possibly depending on the dimension $n$ of the space) should work.

\smallskip

A more challenging task would be generalizing the birth and death mechanism of the particles: the techniques being used in this paper rely on the linearity of the birth rate $b$ and of the death rate $d$. By a domination argument, it is possible to show that whenever one chooses a bounded birth rate, the corresponding process is well defined. The methods for proving tightness can be extended, too, but the identification of the limit becomes a more difficult matter. 
%
A  motivating and realistic generalization would be to consider a death rate $d=d_x(\eta)$ that depends on the amount of particles present at  each given site and $b$ constant. A classical example is the local logistic model with  death rate $d_x(\eta)=d+c\eta(x)$, where $c>0$ is a competition factor. In this case, when the population becomes large  on each site,  we expect the limiting density $\rho_t(x)$ to satisfy a reaction-diffusion equation of the kind $\partial_t \rho\,=\,\sigma\Delta \rho +(b-d-c\rho)\rho$. Otherwise, if the population remains of order of magnitude $1$ on each site, we expect rapid stirring (see e.g.~\cite{D95} or \cite{P00}).  
In this case the limiting reaction coefficient should come from an averaging procedure of local population size in quasi-state stationary law.
 This follows  the spirit of the replacement lemma used for zero-range processes in statistical mechanics, see \cite{KL98}, which, roughly put, allows one to compare the particle density on microscopic boxes with the particle density on macroscopic ones.
Another interesting extension would be to consider births (or deaths) that depend via a  non-local kernel on the population size in a surrounding region. This would mimic the inclination of individuals to reproduce less in crowded communities and would give further importance to the spatial aspect of the particle system.
Again we believe that rapid stirring techniques can be applied to tackle this problem.

\smallskip

Finally, in this paper we have worked with a scaling corresponding to a fast motion, compared to the population dynamics or epidemics. In fact, the jump rates carry an additional $N^2$ factor, while births and deaths do not. It would be relevant  for our motivations in epidemiology (ANR project Cadence) and interesting for the branching process analysis to consider less separated scales.   

\bigskip

{\bf Acknowledgment}
The authors are very grateful to Jerome Coville, Alessandra Faggionato and Chi Tran Viet for stimulating discussions and relevant suggestions.

\smallskip

This work  was partially funded by the ANR Cadence ANR-16-CE32-0007 and
Chair “Mod\'elisation Math\'ematique et Biodiversit\'e” of VEOLIA-Ecole polytechnique-MNHN-F.X and by the MIUR Excellence Department Project MatMod@TOV awarded to the Department of Mathematics, University of Rome Tor Vergata. 

\smallskip

M.S.~also thanks the INdAM unit GNAMPA. 

\vspace{1cm}
\bibliography{bibliografia.bib}
\bibliographystyle{abbrvnat}

\end{document}